\documentclass[11pt]{amsart}

\usepackage{graphics}
\usepackage{color}
\usepackage[a4paper, margin=3cm]{geometry}
\usepackage{array}
\usepackage{amssymb,enumerate}
\usepackage{amsmath}
\usepackage{bbm}           
\usepackage{bm}
\usepackage{eso-pic,graphicx}
\usepackage{tikz}
\usepackage{cite}
\usepackage{esint}
\usepackage[colorlinks=true, pdfstartview=FitV, linkcolor=blue, citecolor=blue, urlcolor=blue]{hyperref}

\usepackage{graphicx}
\usepackage{bbding}
\usepackage{makecell}
\usepackage{amsmath,amsfonts}
\usepackage{rotating}
\usepackage{amssymb}
\usepackage{verbatim}
\usepackage{rotating}
\usepackage{mathrsfs}
\DeclareSymbolFont{largesymbols}{OMX}{cmex}{m}{n}
\makeatletter
\def\Ddots{\mathinner{\mkern1mu\raise\p@
\vbox{\kern7\p@\hbox{.}}\mkern2mu
\raise4\p@\hbox{.}\mkern2mu\raise7\p@\hbox{.}\mkern1mu}}
\makeatother
\newtheorem*{thA}{Theorem A}

\newtheorem*{questionsa}{Questions A}
\newtheorem*{questionsb}{Questions B}
\newtheorem*{questionsc}{Questions C}

\def\XXint#1#2#3{{\setbox0=\hbox{$#1{#2#3}{\int}$}
\vcenter{\hbox{$#2#3$}}\kern-.5\wd0}}

\begin{document}

\newtheorem{definition}{Definition}
\newtheorem{theorem}[definition]{Theorem}
\newtheorem{proposition}[definition]{Proposition}
\newtheorem{conjecture}[definition]{Conjecture}
\def\theconjecture{\unskip}
\newtheorem{corollary}[definition]{Corollary}
\newtheorem{lemma}[definition]{Lemma}
\newtheorem{claim}[definition]{Claim}
\newtheorem{sublemma}[definition]{Sublemma}
\newtheorem{observation}[definition]{Observation}
\theoremstyle{definition}

\newtheorem{notation}[definition]{Notation}
\newtheorem{remark}[definition]{Remark}
\newtheorem{question}[definition]{Question}

\newtheorem{example}[definition]{Example}
\newtheorem{problem}[definition]{Problem}
\newtheorem{exercise}[definition]{Exercise}
 \newtheorem{thm}{Theorem}
 \newtheorem{cor}[thm]{Corollary}
 \newtheorem{lem}{Lemma}[section]
 \newtheorem{prop}[thm]{Proposition}
 \theoremstyle{definition}
 \newtheorem{dfn}[thm]{Definition}
 \theoremstyle{remark}
 \newtheorem{rem}{Remark}
 \newtheorem{ex}{Example}
 \numberwithin{equation}{section}
\def\C{\mathbb{C}}
\def\R{\mathbb{R}}
\def\Rn{{\mathbb{R}^n}}
\def\Rns{{\mathbb{R}^{n+1}}}
\def\Sn{{{S}^{n-1}}}
\def\M{\mathbb{M}}
\def\N{\mathbb{N}}
\def\Q{{\mathbb{Q}}}
\def\Z{\mathbb{Z}}
\def\X{\mathbb{X}}
\def\Y{\mathbb{Y}}
\def\F{\mathcal{F}}
\def\L{\mathcal{L}}
\def\S{\mathcal{S}}
\def\supp{\operatorname{supp}}
\def\essi{\operatornamewithlimits{ess\,inf}}
\def\esss{\operatornamewithlimits{ess\,sup}}

\numberwithin{equation}{section}
\numberwithin{thm}{section}
\numberwithin{definition}{section}
\numberwithin{equation}{section}

\def\earrow{{\mathbf e}}
\def\rarrow{{\mathbf r}}
\def\uarrow{{\mathbf u}}
\def\varrow{{\mathbf V}}
\def\tpar{T_{\rm par}}
\def\apar{A_{\rm par}}

\def\reals{{\mathbb R}}
\def\torus{{\mathbb T}}
\def\t{{\mathcal T}}
\def\heis{{\mathbb H}}
\def\integers{{\mathbb Z}}
\def\z{{\mathbb Z}}
\def\naturals{{\mathbb N}}
\def\complex{{\mathbb C}\/}
\def\distance{\operatorname{distance}\,}
\def\support{\operatorname{support}\,}
\def\dist{\operatorname{dist}\,}
\def\Span{\operatorname{span}\,}
\def\degree{\operatorname{degree}\,}
\def\kernel{\operatorname{kernel}\,}
\def\dim{\operatorname{dim}\,}
\def\codim{\operatorname{codim}}
\def\trace{\operatorname{trace\,}}
\def\Span{\operatorname{span}\,}
\def\dimension{\operatorname{dimension}\,}
\def\codimension{\operatorname{codimension}\,}
\def\nullspace{\scriptk}
\def\kernel{\operatorname{Ker}}
\def\ZZ{ {\mathbb Z} }
\def\p{\partial}
\def\rp{{ ^{-1} }}
\def\Re{\operatorname{Re\,} }
\def\Im{\operatorname{Im\,} }
\def\ov{\overline}
\def\eps{\varepsilon}
\def\lt{L^2}
\def\diver{\operatorname{div}}
\def\curl{\operatorname{curl}}
\def\etta{\eta}
\newcommand{\norm}[1]{ \|  #1 \|}
\def\expect{\mathbb E}
\def\bull{$\bullet$\ }

\def\blue{\color{blue}}
\def\red{\color{red}}

\def\xone{x_1}
\def\xtwo{x_2}
\def\xq{x_2+x_1^2}
\newcommand{\abr}[1]{ \langle  #1 \rangle}

\newcommand{\Norm}[1]{ \left\|  #1 \right\| }
\newcommand{\set}[1]{ \left\{ #1 \right\} }
\newcommand{\ifou}{\raisebox{-1ex}{$\check{}$}}
\def\one{\mathbf 1}
\def\whole{\mathbf V}
\newcommand{\modulo}[2]{[#1]_{#2}}
\def \essinf{\mathop{\rm essinf}}
\def\scriptf{{\mathcal F}}
\def\scriptg{{\mathcal G}}
\def\m{{\mathcal M}}
\def\scriptb{{\mathcal B}}
\def\scriptc{{\mathcal C}}
\def\scriptt{{\mathcal T}}
\def\scripti{{\mathcal I}}
\def\scripte{{\mathcal E}}
\def\V{{\mathcal V}}
\def\scriptw{{\mathcal W}}
\def\scriptu{{\mathcal U}}
\def\scriptS{{\mathcal S}}
\def\scripta{{\mathcal A}}
\def\scriptr{{\mathcal R}}
\def\scripto{{\mathcal O}}
\def\scripth{{\mathcal H}}
\def\scriptd{{\mathcal D}}
\def\scriptl{{\mathcal L}}
\def\scriptn{{\mathcal N}}
\def\scriptp{{\mathcal P}}
\def\scriptk{{\mathcal K}}
\def\frakv{{\mathfrak V}}
\def\v{{\mathcal V}}
\def\C{\mathbb{C}}
\def\D{\mathcal{D}}
\def\R{\mathbb{R}}
\def\Rn{{\mathbb{R}^n}}
\def\rn{{\mathbb{R}^n}}
\def\Rm{{\mathbb{R}^{2n}}}
\def\r2n{{\mathbb{R}^{2n}}}
\def\Sn{{{S}^{n-1}}}
\def\bbM{\mathbb{M}}
\def\N{\mathbb{N}}
\def\Q{{\mathcal{Q}}}
\def\Z{\mathbb{Z}}
\def\F{\mathcal{F}}
\def\L{\mathcal{L}}
\def\G{\mathscr{G}}
\def\ch{\operatorname{ch}}
\def\supp{\operatorname{supp}}
\def\dist{\operatorname{dist}}
\def\essi{\operatornamewithlimits{ess\,inf}}
\def\esss{\operatornamewithlimits{ess\,sup}}
\def\dis{\displaystyle}
\def\dsum{\displaystyle\sum}
\def\dint{\displaystyle\int}
\def\dfrac{\displaystyle\frac}
\def\dsup{\displaystyle\sup}
\def\dlim{\displaystyle\lim}
\def\bom{\Omega}
\def\om{\omega}

\author[J. Tan]{Jiawei Tan}
\address{Jiawei Tan:
School of Mathematical Sciences \\
Beijing Normal University \\
Laboratory of Mathematics and Complex Systems \\
Ministry of Education \\
Beijing 100875 \\
People's Republic of China}
\email{jwtan@mail.bnu.edu.cn}

\author[Q. Xue]{Qingying Xue$^{*}$}
\address{Qingying Xue:
	School of Mathematical Sciences \\
	Beijing Normal University \\
	Laboratory of Mathematics and Complex Systems \\
	Ministry of Education \\
	Beijing 100875 \\
	People's Republic of China}
\email{qyxue@bnu.edu.cn}


\thanks{The authors were partly supported by the National Key R\&D Program of China (No. 2020YFA0712900) and NNSF of China (No. 12271041).
\thanks{$^{*}$ Corresponding author, e-mail address: qyxue@bnu.edu.cn}}

\date{\today}
\title[ THE MULTILINEAR PSEUDO-DIFFERENTIAL OPERATORS  ]
{\bf Quantitative weighted estimates for the multilinear pseudo-differential operators in function spaces}

\begin{abstract}
In this paper, the weighted estimates for multilinear pseudo-differential operators were systematically studied in rearrangement invariant Banach and quasi-Banach spaces. These spaces contain the Lebesgue space, the classical Lorentz space and Marcinkiewicz space as typical examples. More precisely, the weighted boundedness and weighted modular estimates, including the weak endpoint case, were established for multilinear pseudo-differential operators and their commutators.  As applications, we show that the above results also hold for the multilinear Fourier multipliers, multilinear square functions, and a class of multilinear Calder\'{o}n-Zygmund  operators.
\end{abstract}\maketitle

\section{Introduction and main results}

The main purpose of the article is to establish the weighted boundedness of the multilinear pseudo-differential operators in the rearrangement invariant Banach and quasi-Banach spaces, and establish the associated modular inequalities.

The study of the pseudo-differential operators can be traced back to the celebrated works of Kohn and Nirenberg \cite{koh} and H\"{o}rmander \cite{hor2}. As was well known, the pseudo-differential operators play an important role in Harmonic analysis as well as in other fields, such as PDEs and mathematical physics. Particularly, the theory of the pseudo-differential operators has not only been widely used in quantum field and exponential theory, but also played an indispensable role in the first proof of the Atiyah-Singer index theorem \cite{ati}. Further important applications were given in the works of H\"{o}rmander \cite{hor1} and  Fefferman and Kohn \cite{fef}.  A generalized class of pseudo-differential operators related to hypoelliptic equations was introduced in \cite{hor1} and applied to study the existence and regularity of the solutions.  By using the tools of microlocal analysis and the pseudo-differential operators on $\mathbb{R}^3,$ Fefferman and Kohn \cite{fef} obtained the optimal H\"{o}lder estimates for the Bergman projection. The corresponding estimates for the Szeg\"{o} projection on 3-D CR manifolds of finite type were also given in \cite{fef}.

We now consider the boundedness of the pseudo-differential operators on Lebesgue spaces and Hardy spaces. The minimum smoothness and decay conditions assumed on the symbols to guarantee their boundedness in these function spaces are what we are interested in.
To make the statement of known results clearly, we need to introduce some definitions and briefly describe the literature of this problem.

\begin{definition}[\textbf{H\"{o}rmander’s class}]\label{def1.1}
Let $m \in \mathbb{N},$ $r \in \mathbb{R}$ and $0\leq \rho, \delta \leq 1,$ $\sigma$ is a smooth function defined on $\mathbb{R}^n \times \mathbb{R}^{m n}$.
We say $\sigma \in S_{\rho, \delta}^r(n, m)$ if for each triple of multi-indices $\alpha:=\left(\alpha_1, \ldots, \alpha_n\right)$ and $\beta_1, \ldots, \beta_m$, there exists a constant $C_{\alpha, \beta}$ such that
$$
\left|\partial_x^\alpha \partial_{\xi_1}^{\beta_1} \cdots \partial_{\xi_m}^{\beta_m} \sigma(x, \vec{\xi})\right| \leq C_{\alpha, \beta}\left(1+\left|\xi_1\right|+\cdots+\left|\xi_m\right|\right)^{r-\rho \sum_{j=1}^m\left|\beta_j\right|+\delta|\alpha|}.
$$
\end{definition}

\begin{definition}[\textbf{$m$-linear pseudo-differential operator}]\label{def1.2}
Given a symbol $\sigma$, $m \in \mathbb{N},$ the $m$-linear pseudo-differential operator $T_\sigma$ is given by
$$ T_{\sigma}(\vec{f})(x)=\int_{(\mathbb{R}^n)^m} \sigma(x,\vec{\xi}) e^{2 \pi i x\cdot(\xi_1+\cdots \xi_m)} \widehat{f}_1(\xi_1) \cdots \widehat{f}_m(\xi_m) d\vec{\xi},$$
where $\widehat{f}$ is the Fourier transform  of the function $f$  defined by
 $$\widehat{f}(\xi)=\int_{\mathbb{R}^n} f(x) e^{-2\pi i x\cdot \xi} dx.$$
\end{definition}

In the linear case $m = 1,$ H\"{o}rmander \cite{hor3}, Kumano-Go \cite{kum}, Calder\'{o}n and Vaillancourt \cite{cal1}  proved the local and global $L^2$ bounds of pseudo-differential operators in the 1970s. In 1976, Rodino \cite{rod} constructed a symbolic function such that the pseudo-differential operator is not bounded on $L^2$.
Later on, using the method of the almost orthogonality principle, Hounie \cite{hou} gave an equivalence description between the indexes of H\"{o}rmander class and the $L^2$ boundedness of the pseudo-differential operators. As a generalization of the classical H\"{o}rmander class, Coifman and Meyer \cite{coi} studied a class of symbols with Dini type conditions and gave a sufficient and necessary condition  for the $L^p$ boundedness of the pseudo-differential operators. Subsequently, Nagase \cite{nag} and Bourdaud \cite{bou} respectively improved Coifman and Meyer's results under two more general conditions assumed on the symbols. In the endpoint cases, Alvarez and Hounie \cite{alv} demonstrated
 that if $\sigma \in S_{\rho, \delta}^r$ with $ 0<\rho \leq 1,0 \leq \delta<1$ and $r=\frac{n}{2}(\rho-1+\min \{0, \rho-\delta\})$, then $T_\sigma$ is of weak type $(1,1)$ and bounded from $H^1$ to $L^1,$ where $H^1$ denotes the classical Hardy spaces.

Great achievements have been made on the study of weighted boundedness of $T_\sigma$. Let $w$ be a Muckenhoupt $A_p$ weight with $w \in A_{p / 2}(2 \leq p<\infty)$ and $\sigma \in S_{\rho, \delta}^{-n(1-\rho) / 2}(n, 1)$ with $0<\delta<\rho<1,$  the weighted boundedness of $T_\sigma$ on $L^p(w)$ has been obtained by Chanillo and Torchinksy in \cite{cha}. Subsequently, Michalowski, Rule and Staubach \cite{mic} considered the above condition with $\rho =\delta$. Recently, Beltran and Cladek \cite{bel} established a sparse domination of the pseudo-differential operators and consequently obtained a quantitative weighted estimate.

We now turn our attention to the study of multilinear pseudo-differential operators. The theory of bilinear pseudo-differential operators was first originated from the famous works of Coifman and Meyer \cite{coi1}, in which they used the bilinear pseudo-differential operators as a typical model for the Calder\'{o}n commutator. In \cite{coi,coi2}, when  $\rho= 1$, Coifman and Meyer showed that $T_\sigma$ is bounded from $L^{p_1} \times \cdots \times L^{p_m}$ to $L^p$ for symbol $\sigma \in S_{1,0}^0(n, m)$ with $1<p, p_1, \ldots, p_m<\infty$ and $\frac{1}{p_1}+\cdots+\frac{1}{p_m}=\frac{1}{p}.$ Grafakos and Torres \cite{gra} considered the boundedness of multilinear pseudo-differential operators $T_\sigma$ with symbol $\sigma$ belongs to $S_{1,1}^0(n, m).$
There is a fundamental question in this field. It is known that the multilinear Calder\'{o}n-Zygmund operators enjoy almost all the properties of the linear Calder\'{o}n-Zygmund operators, one may ask whether the multilinear pseudo-differential operators also have such properties? However, B\'{e}nyi and Torres in \cite{ben1} showed that there exist symbols in $S_{1,1}^0(n, 2)$ such that $T_\sigma$ fails to be bounded from $L^{p_1} \times L^{p_2}$ to $L^p$ for $1 \leq p_1, p_2, p<\infty$ with $\frac{1}{p_1}+\frac{1}{p_2}=\frac{1}{p}.$

The situation becomes complicated and subtle when $\rho\neq 1$. We summarize these cases and classify the boundedness of pseudo-differential operators according to the parameters values of the symbol. See Table 1 below. We refer the readers to see \cite{li, cao} and the references therein for some other weighted boundedness properties of multilinear pseudo-differential operators and their commutators.

\begin{table}[htb]
\begin{center}
\caption{The symbol of multilinear pseudo-differential operators}
\label{table:1}
\begin{tabular}{|c|c|c|c|}
\hline  parameters   & $L^2\times L^2\rightarrow L^1$ & $L^\infty\times L^\infty\rightarrow L^\infty$ & $L^\infty\times L^\infty\rightarrow BMO$ \\
\hline   $\rho=\delta=r=0$ & \XSolidBrush \cite{ben1}&  & \XSolidBrush \cite{miy}\\
\hline   \makecell[c]{$0\leq\delta\leq \rho\leq1,\delta <1  $\\$r<-n(1-\rho)$ }   & \Checkmark \text{when} $r<-\frac{n(1-\rho)}{2}$ \cite{ben3}  & \Checkmark \cite{ben3}& \Checkmark\\
\hline    \makecell[c]{$\delta =\rho=0 $\\$r>-n$ } & \XSolidBrush \cite{miy}&  &\XSolidBrush \cite{miy}\\
\hline  \makecell[c]{$0\leq\delta\leq \rho\leq1 $\\$r=-n(1-\rho)$ }   & \Checkmark \text{when} $\rho<1, \delta=\rho $ \cite{miy2} &  & \makecell[c]{\Checkmark when \\$0\leq\delta\leq \rho<\frac{1}{2} $\cite{nai} \\ or\\ $0\leq\rho<1, \delta=\rho $ \cite{miy2}} \\
\hline
\end{tabular}
\end{center}
\end{table}

In order to state more known results, we give a brief introduction on sparse domination. In the last decade, the problem of sharp weighted norm estimate has led to a deep understanding of the classical Calder\'{o}n-Zygmund operators in the form of sparse domination. To be more precise, Lerner \cite{ler2, ler3} first showed that any standard Calder\'{o}n-Zygmund operator $T$ which satisfies
a H\"{o}lder-Lipschitz condition can be dominated pointwisely by a finite number of sparse operators in the following way
$$
|T f(x)| \leq \sup _{\mathcal{S}}\mathcal{A}_{\mathcal{S}} f(x),
\text{where} ~
 \mathcal{A}_{\mathcal{S}} f(x)=\sum_{Q \in \mathcal{S}} \frac{1}{|Q|} \int_Q|f| \chi_Q(x),
$$
where $\mathcal{S}$ is a sparse family of cubes (see Section $2.2$ below). This estimate almost immediately leads to a simple proof of the sharp dependence of the constant in the relevant weighted norm inequalities, the $A_2$ conjecture, which has been actively investigated for more than a decade (see \cite{hyt0,ler3,lac}). Soon after, Lerner's techniques were widely used for many other operators as well as spaces (cf. e.g. \cite{cej, con, con1}).

We need to introduce the $m$-linear commutators of pseudo-differential operator.
\begin{definition}[\textbf{$m$-linear commutators of pseudo-differential operator}]\label{def1.3}
For the pseudo-differential operator $T_\sigma$ with symbol $\sigma$ and locally integrable functions $\mathbf{b}=\left(b_1, \ldots, b_m\right)$ with $m \in \mathbb{N},$ the $m$-linear commutator of $T_\sigma$ is defined by

$$
T_{\sigma, \Sigma \mathbf{b}}(\vec{f})(x)=\sum_{j=1}^{m}\left[b_{j}, T_{\sigma}\right]_{j}(\vec{f})(x),
$$
where each term is the commutator of $b_{j}$ and $T_{\sigma}$ in the $j$-th entry of $T_{\sigma}$, that is
$$
\left[b_{j}, T_{\sigma}\right]_{j}(\vec{f})(x)=b_{j}(x) T_{\sigma}(\vec{f})(x)-T_{\sigma}\left(f_{1}, \ldots, b_{j} f_{j}, \ldots, f_{m}\right)(x).
$$
\end{definition}
The commutators given in Definition \ref{def1.3} were originally introduced by P\'{e}rez and Torres \cite{per1} in the study of the $m$-linear Calder\'{o}n-Zygmund operators. In 2009, Lerner et al. \cite{ler1} considered the weighted strong-type as well as sharp weak-type estimates for the $m$-linear commutators of the Calder\'{o}n-Zygmund operators. In the linear case $m = 1,$  for any $r<0$ and $ 0 \leq \delta< \rho <1$, the $L^p$ boundedness of the operators $T_{\sigma, b}$ with $\sigma \in S_{\rho, \delta}^r$ and $b\in \mathrm{B M O}$ has been studied by Chanillo \cite{cha1}. For an overview of this area,  we refer to \cite{hun,tang} and references therein.
It is worth mentioning that in \cite{cao}, Cao, Xue and Yabuta obtained some sharp results on multilinear pseudo-differential operators and their commutators, including local exponential estimates, weighted mixed weak type inequality and sharp weighted estimates for $T_\sigma$ and $T_{\sigma, \Sigma \mathbf{b}}$ by using the method of sparse domination. More recent work can be found in \cite{ci} in the space of generalized form. One of the main result in \cite{cao} is as follows:

\begin{thA}[\cite{cao}]
 Assume that $\sigma \in S_{\rho, \delta}^r(n, m)$ with $\rho, \delta \in[0,1]$ and $r<m n(\rho-1)$. Let $\frac{1}{p}=\frac{1}{p_1}+\cdots+\frac{1}{p_m}$ with $1<p_i<\infty, i=1, \ldots, m$. If $\vec{\omega} \in A_{\vec{p}}$ and $\mathbf{b} \in \mathrm{B M O}^m$, then
$$
\left\|T_\sigma\right\|_{L^{p_1}\left(\omega_1\right) \times \cdots \times L^{p_m}\left(\omega_m\right) \rightarrow L^p\left(\nu_{\vec{\omega}}\right)} \leq c_{n, \vec{p}} \mathscr{N}_{\text {weak }}[\vec{\omega}]_{A_{\vec{p}}}^{\beta(\vec{p})}
$$
and
$$
\begin{array}{r}
\left\|T_{\sigma, \Sigma \mathbf{b}}\right\|_{L^{p_1}\left(\omega_1\right) \times \cdots \times L^{p_m}\left(\omega_m\right) \rightarrow L^p\left(\nu_{\vec{\omega}}\right)} \leq c_{n, \vec{p}} \mathscr{N}_{w e a k}\|\mathbf{b}\|_{\mathrm{B M O}}[\vec{\omega}]_{A_{\vec{p}}}^{2 \beta(\vec{p})} ,\\
\end{array}
$$
 where $\mathscr{N}_{\text {weak }}=\left\|T_\sigma\right\|_{L^1 \times \cdots \times L^1 \rightarrow L^{1 / m, \infty}}, \text { and } \beta(\vec{p})=\max\limits _{1 \leq i \leq m}\left\{1, \frac{p_i^{\prime}}{p}\right\},$ for the definition of
 $\vec{w}$ and $A_{\vec{p}}$ see \cite[p. 1232]{ler1}.
\end{thA}

Inspired by the results as above, it is natural to ask the following questions.
\begin{questionsa}[general Banach space]
Whether multilinear operators, such as multilinear pseudo-differential operators $T_\sigma$ and their commutators $T_{\sigma, \Sigma \mathbf{b}}$, multilinear Calder\'{o}n-Zygmund singular integral operators, are bounded in more general spaces, such as Lorentz spaces, Zygmund spaces? Furthermore, can we find a unified way to deal with these operators and spaces?
\end{questionsa}
\begin{questionsb}[general quasi-Banach space]
Can we generalize the spaces in Theorem A to some quasi-Banach spaces ?
\end{questionsb}
\begin{questionsc}[modular inequalities]
What kinds of modular inequalities do multilinear operators $T_\sigma$ and $T_{\sigma, \Sigma \mathbf{b}}$ enjoy ?
\end{questionsc}

In this article, we give a firm answer to these three questions. Since classical Lorentz spaces and Zygmund spaces are all belonging to the rearrangement invariant Banach function spaces, this kind of function spaces come naturally into our research scope. For the case of the quasi-spaces, accordingly, we call the rearrangement invariant quasi-Banach function spaces. A detailed definition of these two kinds of spaces is given in the next section. The study of such spaces has a long history, in 1955, Lorentz \cite{lor} first gave that a sufficient and necessary condition for Hardy-Littlewood maximal operator $M$ to be bounded on rearrangement invariant Banach function space $\mathbb{X}$ is $p_{\mathbb{X}}>1.$ Afterwards, Boyd \cite{boy} generalized this result to the Hilbert transform $H$, and proved that $H$ is bounded on $\mathbb{X}$ if and only if $1<p_{\mathbb{X}} \leqslant q_{\mathbb{X}}<\infty .$  Here $p_{\mathbb{X}}$ and $q_{\mathbb{X}}$ denote the Boyd indices of $\mathbb{X}$ (see Section 2.3 below). For the weighted case, in 2006, Curbera et al. \cite{cur} obtained the boundedness of Calder\'{o}n-Zygmund operator in the weighted rearrangement invariant quasi-Banach function spaces with the same restrictions on Boyd's index. In \cite{and}, Anderson and Hu obtained the quantitative weighted boundedness of the maximal truncated Calder\'{o}n-Zygmund operators on $\mathbb{X}$ using sparse domination approach. For other works related to these spaces, we refer to \cite{ben0,edm} and references therein.

Now we are ready to state our main results.

\begin{theorem}\label{thm1.1}
 Let $\mathbb{X},\mathbb{X}_i$ be rearrangement invariant Banach function spaces with
$1<p_{\mathbb{X}},p_{\mathbb{X}_i}\leq q_{\mathbb{X}}, q_{\mathbb{X}_i}<\infty$ for $i=1,2\cdots, m .$   Let $\sigma \in S_{\rho,\delta}^r(n,m)$ with $0 \leq \rho,\delta \leq 1, r < mn(\rho-1).$ Suppose that $m$-product operator $P_m$ maps
$\overline{\mathbb{X}}_1\times \cdots \times\overline{\mathbb{X}}_m$ to $\overline{\mathbb{X}},$ then for every
$ w \in A_{\min\limits_{i}{\{p_{\mathbb{X}_i}\}}},$ it holds that
\begin{equation*}
\begin{aligned}
\left\| T_{\sigma}(\vec{f})\right\|_{\mathbb{X}(w)} \lesssim [w]_{A_\infty}  \prod\limits_{i=1}\limits^m [w]_{A_{p_{\mathbb{X}_i}}}^{\frac{1}{p_{\mathbb{X}_i}}} \left\| f_i\right\|_{\mathbb{X}_i(w)}.
\end{aligned}
\end{equation*}
\end{theorem}

\begin{remark}
We note that when $m = 1$ and $\mathbb{X} = L^p$, the result of Theorem \ref{thm1.1} coincides with the main conclusion in \cite{bel}, with a more accurate norm constant. When $m > 1$, Theorem \ref{thm1.1} is also a generalization of Theorem A, see Corollary \ref{cor1.1}.
\end{remark}

For the commutators of multilinear pseudo-differential operators, we have
\begin{theorem}\label{thm1.2}
Let $1<r<\infty, 1\leq i\leq m,$ $\mathbb{X},\mathbb{X}_i$ be rearrangement invariant Banach function spaces with $r<p_{\mathbb{X}},p_{\mathbb{X}_i}\leq q_{\mathbb{X}}, q_{\mathbb{X}_i}<\infty.$   Assume that $\sigma \in S_{\rho,\delta}^l(n,m)$ with $0 \leq \rho,\delta \leq 1, l < mn(\rho-1).$ Suppose that $P_m$ maps
$\overline{\mathbb{X}}_1\times \cdots \times\overline{\mathbb{X}}_m$ to $\overline{\mathbb{X}}$ and $w\in A_{p_0/r}, \vec{b} =\left(b_{1}, \ldots, b_{m}\right)\in \mathrm{BMO}^m,$ then there exist $q_i>1, i=1,2\cdots, m,$  such that

\begin{equation*}
\begin{aligned}
\left\| T_{\sigma,\Sigma \mathbf{b}}(\vec{f})\right\|_{\mathbb{X}(w)} \lesssim  \| \vec{b}\|_{\mathrm{BMO}}\cdot[w]_{A_\infty} \left([w]_{A_\infty}  \prod\limits_{i=1}\limits^m [w]_{A_{p_{\mathbb{X}_i}}}^{\frac{1}{p_{\mathbb{X}_i}}}+ \prod\limits_{i=1}\limits^m [w]_{A_{p_{{\mathbb{X}}_i}/r}}^{\frac{1}{q_ir}}    \right) \prod\limits_{i=1}\limits^m \left\| f_i\right\|_{\mathbb{X}_i(w)},
\end{aligned}
\end{equation*}
where $p_0=\min\limits_{i}{\{p_{\X_i}\}},$ $\| \vec{b}\|_{\mathrm{BMO}}:=\sup\limits_{1 \leq j \leq m}\left\|b_{j}\right\|_{\mathrm{BMO}}.$
\end{theorem}

 When $\mathbb{X} $ is a space of rearrangement invariant quasi-Banach type, we obtain the following quantitative weighted estimates.
\begin{theorem}\label{thm1.3}
 Let $\mathbb{X} $ be a rearrangement invariant quasi-Banach function space with $p$-convex property for some $0<p\leq 1.$ For $i=1,2\cdots m$, let $\mathbb{X}_i$ be rearrangement invariant quasi-Banach function spaces with
$1<p_{\mathbb{X}},p_{\mathbb{X}_i}\leq q_{\mathbb{X}}, q_{\mathbb{X}_i}<\infty$ and each of them is $p_i$-convex for some $0<p_i \leq 1.$ Assume that $\sigma \in S_{\rho,\delta}^r(n,m)$ with $0 \leq \rho,\delta \leq 1, r < mn(\rho-1).$ Suppose that $P_m$ maps
$\overline{\mathbb{X}}_1\times \cdots \times\overline{\mathbb{X}}_m$ to $\overline{\mathbb{X}},$ then for every
$ w \in A_{\min\limits_{i}{\{p_{\mathbb{X}_i}\}}},$

\begin{equation*}
\begin{aligned}
\left\| T_{\sigma}(\vec{f})\right\|_{\mathbb{X}(w)} \lesssim [w]_{A_\infty}^{\frac{1}{p}}  \prod\limits_{i=1}\limits^m [w]_{A_{p_{\mathbb{X}_i}}}^{\frac{1}{p_{\mathbb{X}_i}}} \left\| f_i\right\|_{\mathbb{X}_i(w)}.
\end{aligned}
\end{equation*}
\end{theorem}

\begin{theorem}\label{thm1.4}
Let $1<r<\infty,$ $\mathbb{X} $ be a rearrangement invariant quasi-Banach function space with $p$-convex property  for some $0<p\leq 1.$ For $i=1,2\cdots m$, let $\mathbb{X}_i$ be rearrangement invariant quasi-Banach function spaces with
$r<p_{\mathbb{X}},p_{\mathbb{X}_i}\leq q_{\mathbb{X}}, q_{\mathbb{X}_i}<\infty$ and each of them is $p_i$-convex for some $0<p_i \leq 1.$ Assume that $\sigma \in S_{\rho,\delta}^l(n,m)$ with $0 \leq \rho,\delta \leq 1, l < mn(\rho-1).$ Suppose that $P_m$ maps
$\overline{\mathbb{X}}_1\times \cdots \times\overline{\mathbb{X}}_m$ to $\overline{\mathbb{X}}$ and $w\in A_{p_0/r}, \vec{b} =\left(b_{1}, \ldots, b_{m}\right)\in \mathrm{BMO}^m,$ then there exist $q_i>1, i=1,2\cdots, m,$ such that

\begin{equation*}
\begin{aligned}
\left\| T_{\sigma,\Sigma \mathbf{b}}(\vec{f})\right\|_{\mathbb{X}(w)} \lesssim  \| \vec{b}\|_{\mathrm{BMO}}\cdot[w]_{A_\infty}^{\frac{1}{p}} \left([w]_{A_\infty}^p \prod\limits_{i=1}\limits^m [w]_{A_{p_{\mathbb{X}_i}}}^{\frac{p}{p_{\mathbb{X}_i}}} + \prod\limits_{i=1}\limits^m [w]_{A_{p_{{\mathbb{X}}_i}/r}}^{\frac{p}{q_ir}} \right)^{\frac{1}{p}} \prod\limits_{i=1}\limits^m \left\| f_i\right\|_{\mathbb{X}_i(w)},
\end{aligned}
\end{equation*}
where $p_0=\min\limits_{i}{\{p_{\X_i}\}}.$
\end{theorem}

\begin{remark}
Clearly, when $\sigma \in S_{\rho,\delta}^l(n,m)$ with $0 \leq \rho,\delta \leq 1, l < mn(\rho-1),$ Theorem 4.5 in \cite{mic} is just a special case of Theorem \ref{thm1.2} and \ref{thm1.4}.
\end{remark}

Note that, to deal with the boundedness of operators in some Banach spaces, or even quasi-Banach spaces, one often uses dual space method as one of the basic tools. However, in Harmonic analysis, some estimates are not completely related to Banach space or dual spaces. For example, P\'{e}rez \cite{per0} proved that
$$
\sup _{\lambda>0} \varphi(\lambda) w\left(\{y \in \mathbb{R}^n:|[b, T] f(y)|>\lambda\}\right) \leqslant C \sup _{\lambda>0} \varphi(\lambda) w\left(\{y \in \mathbb{R}^n: M^2 f(y)>\lambda\}\right),
$$
where $T$ is Calder\'{o}n-Zygmund operator, $\varphi(\lambda)=\frac{\lambda}{1+\log ^{+} \lambda^{-1}}$ and $w \in A_{\infty}.$  This inequality is very important in illustrating the endpoint estimates of the commutators. Notice that there is a function $\varphi$ on both sides of the above inequality which is not homogeneous and hence each side in above inequality is not a norm or quasi-norm. Therefore, this type of inequality not only reflects the properties of the operator itself, but also reduces the effect of the dual space method.\par

Based on the above analysis, we present the following weighted modular inequalities for multilinear pseudo-differential operators and their commutators, which are completely new even in the unweighted case.

\begin{theorem}\label{thm1.5}
Assume that $\sigma \in S_{\rho,\delta}^r(n,m)$ with $0 \leq \rho,\delta \leq 1, r < mn(\rho-1).$  Let~$\phi$ be a~$N$-function with sub-multiplicative property. If $1<i_{\phi}<\infty$ then there exist two constants~$\alpha, \beta,$ such that for
every~$1<q<{i_{\phi}}$ and ~$w\in A_q ,$
\begin{equation*}
\int_{\mathbb{R}^n}\phi(|T_{\sigma}(\vec{f})(x)|)w(x)dx\leq C(\phi,w)
 \left(\prod\limits_{i=1}^{m}\int_{\mathbb{R}^n}\phi^m(|f_i(x)|)w(x)dx \right)^{\frac{1}{m}},
\end{equation*}
and if $i_\phi =1, w\in A_1,$
$$\sup\limits_\lambda \phi(\lambda)w\left(\{x\in \mathbb{R}^n: |T_{\sigma}(\vec{f})(x)|>\lambda\}\right)^m\leq C\prod\limits_{i=1}^{m}\int_{\mathbb{R}^n}\phi(|f_i(x)|)w(x)dx,$$
where the definitions of $N$-function, $i_{\phi}$ and $C_1$ are listed in Section 2, and
\begin{equation*}
C(\phi, w)= \begin{cases}[w]_{A_{\infty}}^{1+\alpha C_{1}}, &  \beta[w]_{A_{q}}^{\frac{1}{q}}\leq 2, \\ [w]_{A_{\infty}}^{1+\alpha C_{1}}[w]_{A_{q}}^{\frac{mC_1}{q }}, & \beta[w]_{A_{q}}^{\frac{1}{q}}> 2.\end{cases}
\end{equation*}

\end{theorem}

\begin{theorem}\label{thm1.6}
Let $\sigma \in S_{\rho,\delta}^l(n,m)$ with $0 \leq \rho,\delta \leq 1, l < mn(\rho-1).$  $\vec{b} =\left(b_{1}, \ldots, b_{m}\right)\in \mathrm{BMO}^m.$ Assume that~$\phi$ is a~$N$-function with sub-multiplicative property. For each~$1\leq r < \infty,$ if $r<i_{\phi}<\infty,$ then there exist two constants~$\alpha, C(\phi,w,r)$ such that for
every~$1<q<\frac{i_{\phi}}{r}$ and ~$w\in A_q ,$
\begin{equation*}
\begin{aligned}
\int_{\mathbb{R}^n}\phi(|T_{\sigma,\Sigma \mathbf{b}}(\vec{f})(x)|)w(x)dx\leq &C(\phi,w,r)\| \vec{b}\|_{\mathrm{BMO}}^{1+\alpha C_1}\\
&\times \left(\prod\limits_{i=1}^{m}\int_{\mathbb{R}^n}\phi^m(|f_i(x)|)w(x)dx \right)^{\frac{1}{m}},
\end{aligned}
\end{equation*}
where $C_1$ is the constant in (\ref{ie5.3}).
\end{theorem}

\begin{remark}
We would like to give two comments. First of all, using the fact that $A_p=\bigcup_{1<q<p}A_q$ for all $1<p<\infty, $ we know that the results of Theorem \ref{thm1.5} and \ref{thm1.6} also hold for the endpoint case $w\in A_{i_\phi/r}.$
Secondly, for the special case of $m=1$ and $\phi(t)=t^p$ with $1<p<\infty,$ the result in Theorem \ref{thm1.5} covers the conclusion in \cite[Theorem 1.3 and Corollary 4.2]{bel}.
\end{remark}

The organization of this article is as follows:  In Section 2, we will present some notations and definitions, as well as introduce some important properties which will be used later. The proofs of Theorems \ref{thm1.1} and \ref{thm1.2} will be given in Section 3, which will play a important role in the modular inequalities. The purpose of Section 4 is to prove Theorem \ref{thm1.3} and Theorem \ref{thm1.4}. Section 5 will be devoted to give the proofs of Theorems \ref{thm1.5}-\ref{thm1.6}. Finally, some applications will be given in Section 6.\par
Throughout this paper, we always use $C$ to denote a positive constant, which is independent of the main parameters, but it may vary from line to line. If, for any $a, b \in \mathbb{R}, a \leq C b $ ($a \geq C b,$ respectively), we then write $a \lesssim b $ where $C$ is independent of $a$ and $b$, $(a \gtrsim b,$ respectively).

\section{Preliminary}
First, we present some basic facts for sparse family, Orlicz maximal operators, RIBFS, RIQBFS and  modular inequalities.
\subsection{ Sparse family} Some basic facts on dyadic calculus will be taken from \cite{ler2 , ler3}. We begin with the definition of dyadic lattice.
\begin{definition}\label{def2.1}
	A collection, $\mathcal{D}$ of cubes is said to be a dyadic lattice if it satisfies:
	\begin{enumerate}[(i).]
		\item For any $Q \in \mathcal{D}$ its sidelength $\ell_{Q}$ is of the form $2^{k}, k \in \mathbb{Z}$;
		\item $Q \cap R \in\{Q, R, \emptyset\}$ for any $Q, R \in \mathcal{D}$;
		\item The cubes of a fixed sidelength $2^{k}$ form a partition of $\mathbb{R}^{n}$.
	\end{enumerate}

\end{definition}
In dyadic calculus, Three Lattice Theorem (see \cite[Proposition 5.1]{ler2}) plays an important role. It asserts that there are $3^n$ dyadic lattices $\mathcal{D}^{(j)}$ such that for every cube $Q \subset \mathbb{R}^n$, there is a cube $R \in \mathcal{D}^{(j)}$ for some $j$, for which $3 Q \subset R$ and $|R| \leq 9^n|Q|$.\par
	\vspace{0.1cm}
Based on the Definition \ref{def2.1}, we have the following definition of the sparse family.
\begin{definition}
	Let $\mathcal{D}$ be a dyadic lattice. $\mathcal{S} \subset \mathcal{D}$ is said to be a $\eta$-sparse family with $\eta \in(0,1)$ if for every cube $Q \in \mathcal{S},$
	$$|\bigcup_{P \in \mathcal{S}, P \subsetneq Q} P|\leq (1-\eta) \left| Q \right|.$$
\end{definition}
Equivalently, if we define
$$E(Q)=Q \backslash \bigcup_{P \in \mathcal{S}, P \subsetneq Q} P,$$
then a trivial calculation shows that the sets $E(Q)$ are pairwisely disjoint and $|E(Q)| \geq \eta |Q|$.\par
Given a dyadic lattice $\mathcal{D}$ and a $\eta$-sparse family $\mathcal{S}\subseteq \mathcal{D}$, we define the sparse operator as
$$\mathcal{A}_{r, \mathcal{S}} f(x)= \sum_{Q \in \mathcal{S}}\langle|f|^{r}\rangle_{Q}^{1 / r}\chi_{Q}(x)=\sum_{Q \in \mathcal{S}}\left(\frac{1}{|Q|} \int_{Q}|f(y)|^{r} d y\right)^{\frac{1}{r}} \chi_{Q}(x),$$
where $r>0$ and $\langle|f|^{r}\rangle_{Q}=\frac{1}{|Q|} \int_{Q}|f(y)|^{r} d y$.\par
\par

\subsection{  Young function and Orlicz maximal operators}
We need to recall some basic facts from the theory of Orlicz spaces. For more information and a lively exposition about these spaces, we refer the readers to \cite{rao}.\par
Let $\Phi$ be the set of functions $\phi:[0, \infty) \longrightarrow[0, \infty)$ which satisfies non-negative, increasing and such that $\lim _{t \rightarrow \infty} \phi(t)=\infty $  and $\lim _{t \rightarrow 0} \phi(t)=0 .$ If $\phi \in \Phi$ is convex, we say that $\phi$ is a Young function.
The Orlicz space with respect to the measure $\mu, L_{\phi}(\mu)$, is defined to be the set of measurable functions $f$ such that for some $\lambda>0$,
$$
\int_{\mathbb{R}^n} \phi\left(\frac{|f(x)|}{\lambda}\right) d \mu<\infty.
$$
We can also define the $\phi$-norm of $f$ over a cube $Q$ as
$$
\|f\|_{\phi(\mu), Q}:=\inf \left\{\lambda>0: \frac{1}{\mu(Q)} \int_{Q} \phi\left(\frac{|f(x)|}{\lambda}\right) d \mu \leq 1\right\}.
$$
For simplify, we denote $\|f\|_{\phi, Q}$, if $\mu$ is the Lebesgue measure. And if $\mu=w d x$ is an absolutely continuous measure with respect to the Lebesgue measure we write $\|f\|_{\phi(w), Q}$.\par
It is a simple but important observation that each Young function $\phi$ satisfies the following generalized H\"{o}lder inequality:
$$\frac{1}{\mu(Q)} \int_{Q}|f g| d \mu \leq 2\|f\|_{\phi(\mu), Q}\|g\|_{\bar{\phi}(\mu), Q},$$
where $\bar{\phi}(t)=\sup _{s>0}\{s t-\phi(s)\}$ is the complementary function of $\phi$.

Then we use $M_{\phi}$ to represent the dyadic Orlicz maximal operator defined by:
$$
M_{\phi} f(x):=\sup _{x \in Q, Q \in \mathcal{D}}\|f\|_{\phi, Q}.
$$
	where the supremum is taken over all the dyadic cubes containing $x$ and $\mathcal{D}$ is the given dyadic grid. \par
We will employ several times the following particular examples of maximal operators.

\begin{itemize}
	\item If $\phi(t)=t^{r}$ with $r>1,$ then $M_{\phi}=M_{r}$.
	\item If $\phi(t)=t \log (e+t)^{\alpha}$ with $\alpha>0 ,$ then $\bar{\phi}(t) \simeq e^{t^{1 / \alpha}}-1$ and we denote
	$M_{\phi}=M_{L \log L^{\alpha}}$. We have that $M \leq M_{\phi} \lesssim M_{r}$ for all $1<r<\infty.$ Moreover, it can
	be showed that $M_{\phi} \approx M^{l+1},$ where $\alpha=l \in \mathbb{N}$ and $M^{l+1}$ is $M$ iterated $l+1$ times.
\end{itemize}
\subsection{ RIBFS and RIQBFS}
This subsection is devoted to giving the definitions and properties of rearrangement invariant Banach function spaces, denoted by RIBFS, and rearrangement invariant quasi-Banach function spaces (RIQBFS).  \par
$\circ$ \textbf{Basic definitions of RIBFS.} We write $\mathcal{M}$ for the set of measurable functions on $\left(\mathbb{R}^{n}, d x\right)$ and $\mathcal{M}^{+}$ for the nonnegative ones. A rearrangement invariant Banach norm is a mapping $\rho: \mathcal{M}^{+} \mapsto[0, \infty]$ satisfying that
\medskip
\begin{enumerate}[i).]
	\item $\rho(f)=0 \Leftrightarrow f=0,$ a.e.; $\rho(f+g) \leq \rho(f)+\rho(g) ; \rho(a f)=a \rho(f)$ for $a \geq 0$;
	\item If $0 \leq f \leq g,$ a.e., then $\rho(f) \leq \rho(g)$;
	\item If $f_{n} \uparrow f,$ a.e., then $\rho\left(f_{n}\right) \uparrow \rho(f)$;
	\item If $E$ is a measurable set with $|E|<\infty,$ then $\rho\left(\chi_{E}\right)<\infty,$ and $\int_{E} f d x \leq$ $C_{E} \rho(f),$ for some constant $0<C_{E}<\infty,$ depending on $E$ and $\rho,$ but independent of $f$;
	\item $\rho(f)=\rho(g),$ if $f$ and $g$ are equimeasurable, that is, $d_{f}(\lambda)=d_{g}(\lambda), \lambda \geq 0$ where $d_{f}\left(d_{g}\right.$ respectively) denotes the distribution function of $f$ ($g$ respectively).
\end{enumerate}
\medskip
A rearrangement invariant Banach function space $\X$ is defined by:
$$\mathbb{X}=\left\{f \in \mathcal{M}:\|f\|_{\mathbb{X}}:=\rho(|f|)<\infty\right\}.$$
Let $\mathbb{X}^{\prime}$ be the associated Banach function space of $\mathbb{X}$ in the form that
$$
\mathbb{X}^{\prime}=\left\{f \in \mathcal{M}:\|f\|_{\mathbb{X}^{\prime}}:=\sup \left\{\int_{\mathbb{R}^{n}} f g d x: g \in \mathcal{M}^{+}, \rho(g) \leq 1\right\}<\infty\right\}.
$$
An important fact is that $\mathbb{X}$ is a RIBFS if and only if $\mathbb{X}^{\prime}$ is a $\mathrm{RIBFS}.$
 For any $f \in \mathbb{X}$ and $g \in \mathbb{X}^{\prime},$  it also follows that:
$$
\int_{\mathbb{R}^{n}}|f g| dx \leqslant\|f\|_{\mathbb{X}}\|g\|_{\mathbb{X}^{\prime}},
$$
which can be regarded as the generalized H\"{o}lder's inequality.\par

Another key observation in a RIBFS $\X$ is that the Lorentz-Luxemburg theorem is still true, which states that
$$\|f\|_{\mathbb{X}}=\sup \left\{\left|\int_{\mathbb{R}^{n}} f g d x\right|: g \in \mathbb{X}^{\prime},\|g\|_{\mathbb{X}^{\prime}} \leq 1\right\}.$$\par
Recall the decreasing rearrangement of $f$ is the function $f^{*}$  defined by
$$
f^{*}(t)=\inf \left\{\lambda \geq 0: d_{f}(\lambda) \leq t\right\}, \text{for} ~ t \geq 0.
$$
The main property of $f^{*}$ is that it is equimeasurable with $f$. This allows one to use Luxemburg’s representation theorem to obtain a representation of $\X$. In other words, there exists a RIBFS $\overline{\mathbb{X}}$ over $\left(\mathbb{R}^{+}, d t\right)$ such that $f \in \mathbb{X}$ if and only if $f^{*} \in \overline{\mathbb{X}}$, and in this case $\|f\|_{\mathbb{X}}=\left\|f^{*}\right\|_{\overline{\mathbb{X}}}$.
Moreover, it can be verified that $\overline{\mathbb{X}}^{\prime}=\overline{\mathbb{X}^{\prime}}$ and $\|f\|_{\mathbb{X}^{\prime}}=\left\|f^{*}\right\|_{\overline{\mathbb{X}}^{\prime}}$ for the associated space.\par

\medskip
$\circ$ \textbf{Weighted versions of RIBFS $\X$.}
This paragraph mainly introduces the weighted versions of RIBFS $\X$, before that, we recall the definitions of Muckenhoupt weights. We say that a weight $w$ which is a nonnegative and locally integrable function on $\mathbb{R}^{n}$ belongs to the Muckenhoupt class $A_{p}$ with $1<p<\infty,$ if
$$
[w]_{A_{p}}=\sup _{Q} \left(\frac{1}{|Q|} \int_{Q} w(x) d x\right)\left(\frac{1}{|Q|} \int_{Q} w(x)^{1-p^{\prime}} d x\right)^{p-1} <\infty,
$$
and for $w \in A_{\infty},$
$$
[w]_{A_{\infty}}=\sup _{Q} \frac{1}{w(Q)} \int_{Q} M\left(w \chi_{Q}\right)(x) d x,
$$
where the supremum is taken over all cubes $Q \subset \mathbb{R}^{n}.$\par
A weight $w$ belongs to the class $A_1$ if there is a constant $C$ such that
$$
\frac{1}{|Q|} \int_Q w(y) d y \leqslant C \inf _Q w,
$$
and the infimum of these constants $C$ is called the $A_1$ constant of $w.$
Based on this, we can give the distribution function and the decreasing rearrangement with respect to $w$ as follows
$$w_{f}(\lambda)=w\left(\{x \in \mathbb{R}^{n}:|f(x)|>\lambda\}\right) ; \quad f_{w}^{*}(t)=\inf \left\{\lambda \geqslant 0: w_{f}(\lambda) \leqslant t\right\}.$$
Then, we define the weighted version of the space $\X$,
$$\mathbb{X}(w)=\left\{f \in \mathcal{M}:\|f\|_{\mathbb{X}(w)}:=\left\|f_{w}^{*}\right\|_{\overline{\mathbb{X}}}<\infty\right\}.$$
It is well known that $\mathbb{X}^{\prime}(w)=\mathbb{X}(w)^{\prime}$ (see \cite[p. 168]{cur}).\par

\vspace{0.15cm}
$\circ$ \textbf{ Boyd indices and $r$ exponent.} The Boyd indices play a key role in RIBFS and are closely related to many properties. Let's start with the dilation operator
$$ D_{t} f(s)=f\left(\frac{s}{t}\right), 0<t<\infty, f \in \overline{\mathbb{X}} ,$$
and its norm
$$ h_{\mathbb{X}}(t)=\left\|D_{t}\right\|_{\overline{\mathbb{X}} \mapsto \overline{\mathbb{X}}},  0<t<\infty .$$
We can define the lower and upper Boyd indices $p_{\mathbb{X}}$ and $q_{\mathbb{X}}$, by

$$p_{\mathbb{X}}=\lim _{t \rightarrow \infty} \frac{\log t}{\log h_{\X}(t)}=\sup _{1<t<\infty} \frac{\log t}{\log h_{\X}(t)}, \quad q_{\mathbb{X}}=\lim _{t \rightarrow 0^{+}} \frac{\log t}{\log h_{\mathbb{X}}(t)}=\inf _{0<t<1} \frac{\log t}{\log h_{\mathbb{X}}(t)}.$$
By a trivial calculation, we can see that $1 \leq p_{\mathbb{X}} \leq q_{\mathbb{X}} \leq \infty,$ which follows from the fact that $h_{\mathbb{X}}(t)$ is submultiplicative. Here, we take a special case to illustrate these two Boyd indices. If $\X =L^p$, then $h_{\X}(t)=t^{\frac{1}{p}}$ and thus $p_{\X}=q_{\mathbb{X}}=p.$ \par
The Boyd indices of $\mathbb{X}$ and $\mathbb{X}^{\prime}$ have the following relationships: $p_{\mathbb{X}^{\prime}}=\left(q_{\mathbb{X}}\right)^{\prime}$ and $q_{\mathbb{X}^{\prime}}=\left(p_{\mathbb{X}}\right)^{\prime},$ where $p$ and
$p^{\prime}$ are conjugate exponents.\par

Now, we consider the following $r$ exponent, for each $0<r<\infty$, of RIBFS $\mathbb{X}$:
$$\mathbb{X}^{r}=\left\{f \in \mathcal{M}:|f|^{r} \in \mathbb{X}\right\},$$
with the norm $\|f\|_{\X ^{r}}=\left\||f|^{r}\right\|_{\X}^{\frac{1}{r}}$. It is easy to see that
 $p_{\mathbb{X}^{r}}=p_{\mathbb{X}} \cdot r$ and $q_{\mathbb{X}^{r}}=q_{\mathbb{X}} \cdot r$. \par

\vspace{0.15cm}
$\circ$ \textbf{The case of RIQBFS.}
Samely as $L^p$ space, if $\X$ is a RIBFS,  then $\mathbb{X}^{r}$ is still a RIBFS for $r \geq 1.$ However, in general, the space $\mathbb{X}^{r}$ is not necessarily a Banach space with $0<r<1,$ (see \cite[p. 269]{cur}). Therefore, it is natural to consider the case of quasi-Banach space.\par
We say a mapping $\rho^{\prime}: \mathcal{M}^{+} \mapsto[0, \infty)$ is a rearrangement invariant quasi-Banach function norm if $\rho^{\prime}$ satisfies the basic condition \romannumeral2),\romannumeral3), \romannumeral5) and \romannumeral1) is replaced by the following inequality:
$$\rho^{\prime}(f+g) \leq C\left(\rho^{\prime}(f)+\rho^{\prime}(g)\right),$$
where $C$ is an absolute positive constant. Similarly, a rearrangement invariant quasi-Banach function space (RIQBFS) is a collection of all measurable functions which satisfies $\rho^{\prime}(|f|)<\infty$. In order to build the connection between RIBFS and RIQBFS, we impose the following $p$-convex condition with $0<p<1$ on $\X$ which is a RIQBFS (see \cite[p. 3]{gra1}):
$$\left\|\left(\sum_{j=1}^{N}\left|f_{j}\right|^{p}\right)^{\frac{1}{p}}\right\| _{\mathbb{X}} \lesssim\left(\sum_{j=1}^{N}\left\|f_{j}\right\|_{\mathbb{X}}^{p}\right)^{\frac{1}{p}}.$$

In this case, a basic observation is that the $p$-convex property is equivalent to the fact that $\mathbb{X}^{\frac{1}{p}}$ is a RIBFS.
Thereby, applying the Lorentz-Luxemburg's theorem, we can give a norm inscription of the RIQBFS, that is

$$\|f\|_{\mathbb{X}} \simeq \sup \left\{\left(\int_{\mathbb{R}^{n}}|f(x)|^{p} g(x) d x\right)^{\frac{1}{p}}: g \in \mathcal{M}^{+},\|g\|_{\mathbb{Y}^{\prime}} \leq 1\right\},$$
where $\mathbb{Y}^{\prime}$ is the associated space of the RIBFS $\mathbb{Y}=\mathbb{X}^{\frac{1}{p}}$. Following the previous definition, we note that if $\X$ is a RIQBFS, then $p_{\mathbb{X}^{r}}=p_{\mathbb{X}} \cdot r$ and equivalently for $q_\mathbb{X}$. In a similar fashion, we can define $\mathbb{X}(w)$ for a RIQBFS $\X$ and it enjoys that $\mathbb{X}(w)^{r}=\mathbb{X}^{r}(w)$ with $w \in A_{\infty}$ and $0<r<\infty$, (see \cite[p. 269]{cur}).

\subsection{  Modular inequality}
In this subsection, we review some concepts related to Young functions and modular inequalities. Let's start with the definition of quasi-convex.
A function $\phi \in \Phi$ is said to be quasi-convex means that there exist a convex function $\widetilde{\phi}$ and $a_{1} \geqslant 1$ such that
$$\widetilde{\phi}(t) \leq \phi(t) \leq a_{1} \widetilde{\phi}\left(a_{1} t\right), \quad t \geq 0.$$
Given a positive increasing function $\phi,$ similar to the Boyd indices, we can also define the lower and upper dilation indices of $\phi,$ respectively, as follows:
$$i_{\phi}=\lim _{t \rightarrow 0^{+}} \frac{\log h_{\phi}(t)}{\log t}=\sup _{0<t<1} \frac{\log h_{\phi}(t)}{\log t}, \quad I_{\phi}=\lim _{t \rightarrow \infty} \frac{\log h_{\phi}(t)}{\log t}=\inf _{1<t<\infty} \frac{\log h_{\phi}(t)}{\log t},$$
where $h_{\phi}(t)=\sup _{s>0} \frac{\phi(s t)}{\phi(s)}, t>0.$
As mentioned earlier, the dilation indices have a relationship with Boyd indices. For instance, if $\X$ is the Orlicz space induced by $\phi$, then $p_{\mathbb{X}}=i_{\phi}$ and $q_{\mathbb{X}}=I_{\phi}$ (see \cite[p. 274]{cur}).\par
Now we turn to the $\Delta_{2}$ condition. A function $\phi \in \Phi$ satisfies the $\Delta_{2}$ condition  if $\phi$ is doubling (we write $\phi \in \Delta_{2}$).
A crucial fact is that if $\phi$ is quasi-convex, then $i_{\phi} \geqslant 1$ and that $\phi \in \Delta_{2}$ if and only if $I_{\phi}<\infty.$ Moreover, $\bar{\phi} \in \Delta_{2}$ if and only if $i_{\phi}>1,$ where $\bar{\phi}$ is the complementary function of $\phi$ defined in the second subsection.\par
\vspace{0.1cm}
Let $w \in A_{\infty}$ and $\phi \in \Phi$, we define the modular of $f$ as
$$\rho_{w}^{\phi}(f)=\int_{\mathbb{R}^{n}} \phi(|f(x)|) w(x) d x.$$
The collection of functions
$$
Z_{w}^{\phi}=\left\{f: \rho_{w}^{\phi}(f)<\infty\right\}
$$
is called as a modular space.
A multilinear operator $T$ is said to satisfy a modular inequality on $Z_{w}^{\phi}$ if there exist constants $c_{i}^{1}, c_{i}^2>0$ with $i=1,\ldots, m,$ such that
$$\rho_{w}^{\phi}(T \vec{f}) \leq \prod _{i=1}^m c_{i}^{1} \rho_{w}^{\phi}\left(c_{i}^{2} f_i\right).$$\par

\vspace{0.2cm}

\section{ Proofs of Theorems \ref{thm1.1} and \ref{thm1.2}}

This section is devoted to prove Theorem \ref{thm1.1} and Theorem \ref{thm1.2}. To this aim, let us give some notations and lemmas.
Given a sparse family $\mathcal{S}$, we can define multilinear sparse operators $\mathcal{A}_{\mathcal{S}},$ $\mathcal{A}_{\mathcal{S}, b_j}$ and $\mathcal{A}_{\mathcal{S}, b_j}^*$ with $1\leq j\leq m$ as follows,
$$
\begin{aligned}
\mathcal{A}_{\mathcal{S}}(\vec{f})(x) &:=\sum_{Q \in \mathcal{S}} \prod_{i=1}^m\left\langle\left|f_i\right|\right\rangle_Q \chi_Q(x); \\
\mathcal{A}_{\mathcal{S}, b_j}(\vec{f})(x) &:=\sum_{Q \in \mathcal{S}}\left|b_j(x)-b_{j, Q}\right|\left\langle\left|f_j\right|\right\rangle_Q \prod_{i \neq j}\left\langle\left|f_i\right|\right\rangle_Q \chi_Q(x)
\end{aligned}
$$
and
$$
\mathcal{A}_{\mathcal{S}, b_j}^*(\vec{f})(x):=\sum_{Q \in \mathcal{S}}\left\langle\left|\left(b_j-b_{j, Q}\right) f_j\right|\right\rangle_Q \prod_{i \neq j}\left\langle\left|f_i\right|\right\rangle_Q \chi_Q(x) .
$$

The following results of sparse domination (see \cite[Proposition 4.1]{cao}) play a crucial role in our analysis and will be used several times later.
\begin{lemma}[\cite{cao}]\label{lem1.1}
 Let $\sigma \in S_{\rho, \delta}^r(n, m)$ with $0 \leq \rho, \delta \leq 1$ and $r<m n(\rho-1)$. Then, for every compactly supported functions $f_i, i=1, \ldots, m$, there exist $3^n+1$ sparse collections $\mathcal{S}$ and $\left\{\mathcal{S}_i\right\}_{i=1}^{3^n}$ such that
$$
\left|T_\sigma(\vec{f})(x)\right| \lesssim \mathcal{A}_{\mathcal{S}}(\vec{f})(x), \quad \text { a.e. } x \in \mathbb{R}^n,
$$
and
$$
\left|T_{\sigma, \Sigma \mathbf{b}}(\vec{f})(x)\right| \lesssim \sum_{i=1}^{3^n} \sum_{j=1}^m\left(\mathcal{A}_{\mathcal{S}_i, b_j}(\vec{f})(x)+\mathcal{A}_{\mathcal{S}_i, b_j}^*(\vec{f})(x)\right), \quad \text { a.e. } x \in \mathbb{R}^n.
$$
\end{lemma}

 We also need the following Lemma in \cite[Lemma 3.3]{and}.
\begin{lemma}[\cite{and}]\label{lem1.2}
Let $\mathbb{X}$ be an RIQBFS which is $p$-convex for some $0<p \leq 1$. If $1<p_{\mathbb{X}}<\infty$, then for all $w \in A_{p_{\mathbb{X}}},$ we have
$$
\|M\|_{\mathbb{X}(w) \mapsto \mathbb{X}(w)} \leq C[w]_{A_{p_{\mathbb{X}}}}^{1 / p_{\mathbb{X}}},
$$
where $C$ is an absolute constant only depending on $p_{\mathbb{X}}$ and $n$.
\end{lemma}

\begin{proof}[Proof of Theorem $\ref{thm1.1}$]

We define the $m$-product operator $P_m$ by
$$
P_m(\vec{f})(x)=P_m\left(f_1, f_2, \ldots, f_m\right)(x)=\prod_{j=1}^m f_j(x).
$$

Recall that for any $1<p<\infty, A_p=\bigcup\limits_{q\in (1,p)}A_q.$ Therefore, each $w\in A_{\underset i{\min}{\{p_{\X_i}\}}}$ \\
implies that $w\in A_{p_{\X_i}}$ with $1\leq i\leq m,$ and $[w]_{A_{p_{\X_i}}} \leq [w]_{A_{\underset i{\min}{\{p_{\X_i}\}}}} <\infty.  $  \\
By the definition of $\X$ and Lemma \ref{lem1.1}, we have
\begin{equation}\label{ie3.1}
\begin{aligned}
{\left\|T_\sigma (\vec{f})\right\|}_ {\X(w)}&= \sup _{\|g\|_{\X^{\prime}(w)} \leq 1}\left|\int_{\mathbb{R}^{n}} T_\sigma (\vec{f})(x) g(x) w(x) d x\right| \\
&\lesssim \sup _{\|g\|_{\X^{\prime}(w)}\leq 1}\sum_{Q \in \mathcal{S}}\prod_{i=1}^m {\langle |f_i |\rangle}_Q \int_Q |g(x)|w(x)dx  \\
&\leq \sup _{\|g\|_{\X^{\prime}(w)}\leq 1} \sum_{Q \in \mathcal{S}}
\left(\frac{1}{w(Q)} \int_Q \left(\mathcal{M}(\vec{f})(x)\right)^{\frac{1}{2}}
\left( M_w^{\mathcal{D}} g(x) \right)^{\frac{1}{2}} w(x)dx \right)^2 w(Q).
\end{aligned}
\end{equation}
Assume that $\mathcal{S}$ is a $\eta$-sparse family (i.e., $|E(Q)|\geq \eta |Q|$ holds for any $Q\in \mathcal{S}$) where $\eta$ is an absolute constant depending only on $n.$  Therefore, for each dyadic cube $R \in \mathcal{S},$ it holds that
\begin{equation}\label{ie3.11}
\begin{aligned}
\sum_{Q \subseteq R}w(Q)&=\sum_{Q \subseteq R} \frac{w(Q)}{|Q|}|Q|\leq \frac{1}{\eta}\sum_{Q \subseteq R}\frac{w(Q)}{|Q|}|E(Q)|  \\
&\leq \frac{1}{\eta}\sum_{Q \subseteq R} \int_{E(Q)} M(w\chi_Q)(x)dx   \\
&\leq \frac{1}{\eta}\int_{R} M(w\chi_R)(x)dx     \\
&\leq \frac{1}{\eta} w(R) [w]_{A_\infty}.
\end{aligned}
\end{equation}\par
Combining inequality (\ref{ie3.11}) with the Carleson embedding theorem (\cite[Theorem 4.5]{hyt}) and using the generalized H\"{o}lder's inequality together with (\ref{ie3.1}), it may lead to
\begin{equation*}
\begin{aligned}
{\left\|T_\sigma (\vec{f})\right\|}_ {\X(w)}& \lesssim\frac{4}{\eta}[w]_{A_\infty} \sup _{\|g\|_{\X^{\prime}(w)} \leq 1} \int_{\mathbb{R}^n} \mathcal{M}(\vec{f})(x)  M_w^{\mathcal{D}}  g(x) w(x) dx \\
&\lesssim [w]_{A_\infty} \left\|\mathcal{M}(\vec{f})\right\|_{\X(w)} \sup _{\|g\|_{\X^{\prime}(w)} \leq 1} {\left\| M_w^{\mathcal{D}} g \right\|}_{\X'(w)}   \\
&\leq [w]_{A_\infty} {\left\|P_m((Mf_1,\ldots,Mf_n))\right\|}_{\X(w)},
\end{aligned}
\end{equation*}
where the last inequality follows from the boundedness of $M_w^{\mathcal{D}}$ in \cite{and} or \cite[Theorem 2.3]{cur}.\par
Finally, Lemma \ref{lem1.2}, together with the boundedness of $P_m$ and \cite[Corollary 6.11]{cur}, implies that
\begin{equation}
\begin{aligned}
{\left\|T_\sigma (\vec{f})\right\|}_ {\X(w)}& \lesssim  [w]_{A_\infty} \prod_{i=1}^m {\left\|Mf_i\right\|}_{\X_i(w)} \lesssim [w]_{A_\infty} \prod_{i=1}^m [w]^{\frac{1}{p_{\X_i}}}_{A_{p_{\X_i}}} {\left\|f_i\right\|}_{\X_i(w)}.
\end{aligned}
\end{equation}
Then we finish the proof of Theorem \ref{thm1.1}.
\end{proof}\par
As a corollary of Theorem \ref{thm1.1}, we obtain the following result.

\begin{corollary}\label{cor1.1}
Assume that $\sigma \in S_{\rho, \delta}^r(n, m)$ with $0 \leq \rho, \delta \leq 1$ and $r<m n(\rho-1)$.
Let $\frac{1}{p}=\frac{1}{p_1}+\cdots+\frac{1}{p_m}$, with $1<p_1,\cdots, p_m < \infty.$ If $w \in A_{\underset i{\min}{\{p_i\}}},$ then $T_\sigma $ is bounded from  $L^{p_1}(w)\times \cdots \times L^{p_m}(w)$ to $ L^{p}(w).$
\end{corollary}
\begin{proof}
The proof of Corollary \ref{cor1.1} follows from the same scheme as in the proof of \cite[Corollary 6.11]{cur}, and hence we outline it briefly.\\
 According to Theorem \ref{thm1.1}, when $\mathbb{X}=L^p, \mathbb{X}_i=L^{p_i}$ with $i=1,\ldots, m,$ it suffices to show that for $\frac{1}{p}=\frac{1}{p_1}+\cdots+\frac{1}{p_m}$, $P_m$ is bounded from $\overline{\mathbb{X}}_1 \times \cdots \times \overline{\mathbb{X}}_m$ to $\overline{\mathbb{X}}.$\\

 For $i=1,\ldots, m,$ we take and fix non-negative functions $f_i\in \overline{\mathbb{X}}_i.$ Then it is clear that
$$
\begin{aligned}
\left\|P_m(\vec{f})\right\|_{\overline{\mathbb{X}}}&=\sup_{\|g\|_{{\overline{\mathbb{X}}}^{'}}\leq 1}\left|\int_0^{\infty}  g(t)\prod_{i=1}^m f_i(t) d t\right| \\
& \leqslant \sup_{\|g\|_{{\overline{\mathbb{X}}}^{'}}\leq 1}\int_0^{\infty}  \prod_{i=1}^m (f_i(t)|g(t)|^{\frac{p}{p_i}}) d t\\
& \leqslant \sup_{\|g\|_{{\overline{\mathbb{X}}}^{'}}\leq 1} \prod_{i=1}^m\left(\int_0^{\infty} f_i^{\frac{p_i}{p}}(t) |g(t)| d t\right)^{\frac{p}{p_i}} \\
& \leqslant \sup_{\|g\|_{{\overline{\mathbb{X}}}^{'}}\leq 1} \prod_{i=1}^m \|f_i^{\frac{p_i}{p}}\|^{\frac{p}{p_i}}_{{\overline{\mathbb{X}}}}  \|g\|^{\frac{p}{p_i}}_{{\overline{\mathbb{X}}}^{'}}
=\prod_{i=1}^m\left\|f_i\right\|_{\overline{\mathbb{X}}_i}.
\end{aligned}
$$
For the last inequality, it is enough to show that for any $0<r<\infty,$ $\|f\|_{\overline{{\mathbb{X}}^{r}}}=\|f\|_{\overline{{\mathbb{X}}}^{r}}.$ In fact,  let $g^*=f,$ by the fact that $\|g\|_{\mathbb{X}}=\|g^*\|_{\overline{{\mathbb{X}}}},$ we deduce that $$\|f\|_{\overline{{\mathbb{X}}^{r}}}=\|g^*\|_{\overline{{\mathbb{X}}^{r}}}=\|g\|_{\mathbb{X}^r}=
\||g|^r\|_{\mathbb{X}}^{\frac{1}{r}}=\|(|g|^r)^*\|_{\overline{\mathbb{X}}}^{\frac{1}{r}},$$
which, together with the property in \cite[P.50]{gra1}, further implies that
 $$\|f\|_{\overline{{\mathbb{X}}^{r}}}=\|(|g|^r)^*\|_{\overline{\mathbb{X}}}^{\frac{1}{r}}=
 \|(g^*)^r\|_{\overline{\mathbb{X}}}^{\frac{1}{r}}=\|g^*\|_{\overline{{\mathbb{X}}}^{r}}=
 \|f\|_{\overline{{\mathbb{X}}}^{r}}.$$
\end{proof}
\begin{remark}
Note that Corollary \ref{cor1.1} is a single-weight version of the main results in \cite{cao}. The condition that $P_m$ is bounded in the Theorem \ref{thm1.1} is natural, since it holds automatically when we return back to the classical spaces of $L^{p_1}\times \cdots \times L^{p_m} $, and the method is still valid for the classical Lorentz spaces. However, when dealing with the more general rearrangement invariant Banach function spaces, we do need this condition to deal with the multilinear maximal function.
\end{remark}

As another consequence of Theorem \ref{thm1.1}, we have the following corollary.
\begin{corollary}
Assume that $\sigma \in S_{\rho,\delta}^r(n,m)$ with $0 \leq \rho,\delta \leq 1, r < mn(\rho-1).$ Let $\mathbb{X},\mathbb{X}_i$ be rearrangement quasi-invariant Banach function spaces with $p$-convex for some $p>0,$ and satisfy $1<p_{\mathbb{X}},p_{\mathbb{X}_i}\leq q_{\mathbb{X}}, q_{\mathbb{X}_i}<\infty$ for $i=1,2\cdots, m .$  Assume that $m$-product operator $P_m$ maps
$\overline{\mathbb{X}}_1\times \cdots \times\overline{\mathbb{X}}_m$ to $\overline{\mathbb{X}},$ then for every
$ w \in A_{\min\limits_{i}{\{p_{\mathbb{X}_i}\}}/p},$
\begin{equation*}
\begin{aligned}
\left\| |T_{\sigma}(\vec{f})|^\frac{1}{p}\right\|_{\mathbb{X}(w)} \lesssim [w]_{A_\infty}^\frac{1}{p}  \prod\limits_{i=1}\limits^m [w]_{A_{p_{\mathbb{X}_i}/p}}^{\frac{1}{pp_{\mathbb{X}_i}}} \left\|| f_i|^\frac{1}{p}\right\|_{\mathbb{X}_i(w)}.
\end{aligned}
\end{equation*}
\end{corollary}
\begin{proof}
This corollary follows easily from the fact that $\mathbb{X}^{\frac{1}{p}}$ is a RIBFS and Boyd indices satisfy $p_{\mathbb{X}^{\frac{1}{p}}}=\frac{p_{\X}}{p}.$
\end{proof}

In order to prove Theorem \ref{thm1.2}, we need the following lemma which can be found in \cite[Lemma 3.1]{cao1}.
 \begin{lemma}[\cite{cao1}]\label{lem1.3}
 Let $s>1, t>0$, and $\omega \in A_{\infty}$. Then there holds

$$
\begin{aligned}
&\left\|\omega^{1 / s}\right\|_{L^s(\log L)^{s t}, Q} \lesssim[\omega]_{A_{\infty}}^t\langle\omega\rangle_Q^{1 / s},\\
&\|f \omega\|_{L(\log L)^t, Q} \lesssim[\omega]_{A_{\infty}}^t \inf _{x \in Q} M_\omega\left(|f|^s\right)(x)^{1 / s}\langle\omega\rangle_Q. \\
\end{aligned}
$$
\end{lemma}

\begin{proof}[Proof of Theorem $\ref{thm1.2}$]
By Lemma \ref{lem1.1} and the definition of $\X(w),$ we can control the boundedness of $T_{\sigma,\Sigma \mathbf{b}}$ by
\begin{equation}\label{eq3}
\begin{aligned}
{||T_{\sigma,\Sigma \mathbf{b}}(\vec{f})||}_{\X{(w)}}&=\sup _{\|g\|_{\X^{\prime}(w)} \leq 1} \left|\int_{\mathbb{R}^n} T_{\sigma,\Sigma \mathbf{b}}(\vec{f})(x)g(x)w(x)dx \right|   \\
&\leq \sum_{i=1}^{3^n}\sum_{j=1}^{m} \left( \sup _{\|g\|_{\X^{\prime}(w)} \leq 1} \int_{\mathbb{R}^n} \mathcal{A}_{\mathcal{S}_i,b_j}(\vec{f})(x)|g(x)|w(x)dx\right.\\
&\left.\qquad+ \sup _{\|g\|_{\X^{\prime}(w)} \leq 1}\int_{\mathbb{R}^n} \mathcal{A}_{\mathcal{S}_i,b_j}^*(\vec{f})(x)|g(x)|w(x)dx \right) \\
&=: \sum_{i=1}^{3^n}\sum_{j=1}^{m} \left(\sup _{\|g\|_{\X^{\prime}(w)} \leq 1}\mathcal{I}_{i,j}^{1}(g)+\sup _{\|g\|_{\X^{\prime}(w)} \leq 1} \mathcal{I}_{i,j}^{2}(g)  \right), \\
\end{aligned}
\end{equation}
where
\begin{equation*}
\begin{aligned}
 \mathcal{I}_{i,j}^{1}(g)=\int_{\mathbb{R}^n}\mathcal{A}_{\mathcal{S}_i,b_j}(\vec{f})(x)|g(x)|w(x)dx, \\ \mathcal{I}_{i,j}^{2}(g)=\int_{\mathbb{R}^n} \mathcal{A}_{\mathcal{S}_i,b_j}^*(\vec{f})(x)|g(x)|w(x)dx, \\
\end{aligned}
\end{equation*}
with $1\leq j\leq m, 1\leq i\leq 3^n.$\par
First, we consider the contribuion of $\mathcal{I}_{i,j}^{1}(g).$
Via the definition of $\mathcal{A}_{\mathcal{S}_i,b_j},$ the fact that $||b-b_Q||_{\exp L,Q}\lesssim ||b||_{\mathrm{BMO}}$ (see \cite[Lemma 4.7]{iba}) and Lemma \ref{lem1.3}, the generalized H\"{o}lder's inequality tells us that for any $s>1,$
\begin{equation}\label{eq10}
\begin{aligned}
\mathcal{I}_{i,j}^{1}(g)&=\sum_{Q\in \mathcal{S}_i} \int_Q \left|b_j(x)-b_{j,Q}\right||g(x)|w(x)\prod_{i=1}^m  {\langle \left|f_i\right|\rangle}_Qdx   \\
& \leq 2\sum_{Q\in \mathcal{S}_i} |Q| {||b_j-b_{j,Q}||}_{\exp L,Q} ||gw||_{L(\log L),Q}\prod_{i=1}^m\langle|f_i|\rangle_Q  \\
&\lesssim [w]_{A_\infty}\left\|b_{j}\right\|_{\mathrm{BMO}}\sum_{Q\in \mathcal{S}_i} |Q| \underset{x\in Q}{\inf} \left(M_w^{\mathcal{D}}(|g|^s)(x)\right)^{\frac{1}{s}}{\langle w\rangle}_Q\prod_{i=1}^m  \langle|f_i|\rangle_Q.
\end{aligned}
\end{equation}
Therefore, an elementary calculation, together with Theorem \ref{thm1.1}, yields that
 \begin{equation}\label{eq4}
\begin{aligned}
\mathcal{I}_{i,j}^{1}(g)&\lesssim [w]_{A_\infty}\left\|b_{j}\right\|_{\mathrm{BMO}} \int_{\mathbb{R}^n}\left(M_w^{\mathcal{D}}(|g|^s)\right)^{\frac{1}{s}}(x) \mathcal{A}_{\mathcal{S}_i}(\vec{f})(x)w(x)dx  \\
&\leq [w]_{A_\infty}\left\|b_{j}\right\|_{\mathrm{BMO}} \|\mathcal{A}_{\mathcal{S}_i}(\vec{f})\|_{\mathbb{X}(w)}
  \|\left(M_w^{\mathcal{D}}(|g|^s)\right)^{\frac{1}{s}}\|_{\mathbb{X}^{\prime}(w)}  \\
&\lesssim [w]_{A_\infty}^2\left\|b_{j}\right\|_{\mathrm{BMO}}\|
|g|^s\|^{\frac{1}{s}}_{{\mathbb{X}^{\prime}}^{\frac{1}{s}}(w)}\prod_{i=1}^m [w]^{\frac{1}{p_{\X_i}}}_{A_{p_{\X_i}}} {\left\|f_i\right\|}_{\X_i(w)},
\end{aligned}
\end{equation}
where in the last inequality, we have used $1<p_{\X}\leq q_{\X}<\infty$ and pick $1<s< p_{\X ^{\prime}}=\frac{q_{\X}}{q_{\X}-1}$ such that $p_{(\X^{\prime})^{1/s}}>1.$\par
Now we turn to the proof of $\mathcal{I}_{i,j}^{2}(g).$ For any $1<r<\infty,$ we apply H\"{o}lder's inequality, together with the fact that
$\langle\left|b-b_{Q}\right|^{s}\rangle_{Q}^{1 / s}\leq 2^{n+1}e^{3}  s\|b\|_{\mathrm{BMO}}$ (see \cite[p. 19]{wen}), to obtain
\begin{equation*}
		\begin{aligned}
			\mathcal{I}_{i,j}^{2}(g)&= \int_{\mathbb{R}^n} \mathcal{A}_{\mathcal{S}_i,b_j}^*(\vec{f})(x)|g(x)|w(x)dx\\
			&\leq\sum_{Q \in \mathcal{S}_{i}}\left\langle\left|f_j\left(b_j- b_{j,Q}\right)\right|\right\rangle_{Q} \int_{Q} |g(x)| w(x) d x\prod_{i \neq j}\left\langle\left|f_i\right|\right\rangle_Q\\
			&\leq \sum_{Q \in \mathcal{S}_{i}}\langle\left|b_j-b_{j,Q}\right|^{r^{\prime}}\rangle_{Q}^{1 /r^{\prime}} \int_{Q} |g(x)| w(x) d x\langle\left|f_j\right|^{r}\rangle_{Q}^{1/r}\prod_{i \neq j}\left\langle\left|f_i\right|^r\right\rangle_Q^{1/r}\\
			&\leq 2^{n+1}e^{3} r^{\prime} \|b_j\|_{\mathrm{BMO}} \sum_{Q \in \mathcal{S}_{i}}\frac{1}{w(Q)} \int_{Q} |g(x)| w(x) d x w(Q) \prod_{i=1}^m\left\langle\left|f_i\right|^r\right\rangle_Q^{\frac{1}{r}}.\\
		\end{aligned}
	\end{equation*}
Let $\mathcal{M}_r$ be the multilinear maximal operator with power $r$ defined by
$$\mathcal{M}_r(\vec{f})(x):=\sup\limits _{Q \ni x} \prod\limits_{i=1}^m \left(\frac{1}{|Q|} \int_Q\left|f_i(y)\right|^r d y\right)^{\frac{1}{r}}.$$
We note that each $\mathcal{S}_i$ is a $\frac{1}{2\cdot12^n}$-sparse family (see \cite[p. 1227]{cao1}) and for each $w\in A_{p_0/r}$,  $M_r$ is bounded on $ \X(w)$ where $ M_r(f)=M\left(|f|^r\right)^{1 / r}$ and $M$ is the Hardy-Littlewood maximal operator (see \cite[Lemma 3.1]{tan}). Therefore, by the Carleson embedding theorem, there exist constants $q_i>1 (i=1,\ldots,m)$ such that
\begin{equation}\label{eq5}
		\begin{aligned}
			\mathcal{I}_{i,j}^{2}(g)&\lesssim [w]_{A_{\infty}}\|b_j\|_{\mathrm{BMO}}\int_{\mathbb{R}^{n}} \mathcal{M}_r(\vec{f})(x) M_{w}^{\mathcal{D}} g(x) w(x) d x\\
			&\lesssim [w]_{A_{\infty}}\|b_j\|_{\mathrm{BMO}}\left\|\mathcal{M}_r(\vec{f})\right\|_{\X(w)} \left\|M_{w}^{\mathcal{D}} g\right\|_{\X ^{\prime}(w)}\\
			&\lesssim [w]_{A_{\infty}}\|b_j\|_{\mathrm{BMO}}\left\|P_m\left(M_r(f_1),\ldots,M_r(f_m) \right)\right\|_{\X(w)} \left\| g\right\|_{\X ^{\prime}(w)}\\
&\lesssim [w]_{A_{\infty}}\|b_j\|_{\mathrm{BMO}}\left\| g\right\|_{\X^{\prime}(w)}
\prod_{i=1}^m\left\|M_r(f_i) \right\|_{\X_i(w)} \\
&\lesssim [w]_{A_{\infty}}\|b_j\|_{\mathrm{BMO}}\left\| g\right\|_{\X^{\prime}(w)}
\prod_{i=1}^m[w]^{\frac{1}{q_ir}}_{A_{p_{\X_i}/{r}}}\| f_i\|_{{\X_i}(w)}.\\
		\end{aligned}
	\end{equation}\par
  Finally, taking the supremum over all $g\in \X^{\prime}(w)$ with $\|g\|_{\X^{\prime}(w)}\leq 1$ in (\ref{eq4}) and (\ref{eq5}), together with the estimate of (\ref{eq3}), it yields that

\begin{equation}
\begin{aligned}
{||{T}_{\sigma,\Sigma \mathbf{b}}(\vec{f})||}_{\X{(w)}}&\lesssim\sum_{i=1}^{3^n}\sum_{j=1}^{m} \left(\sup _{\|g\|_{\X^{\prime}(w)} \leq 1}\mathcal{I}_{i,j}^{1}(g)+\sup _{\|g\|_{\X^{\prime}(w)} \leq 1} \mathcal{I}_{i,j}^{2}(g)  \right)   \\
&\lesssim [w]_{A_\infty}\sum_{i}\sum_{j} \|b_j\|_{\mathrm{BMO}}\left(  \|\mathcal{A}_{\mathcal{S}_i}(\vec{f})\|_{\mathbb{X}(w)} +\prod_{i=1}^m\left\|M_r(f_i) \right\|_{\X_i(w)}\right)  \\
&\lesssim \| \vec{b}\|_{\mathrm{BMO}}[w]_{A_\infty} \left([w]_{A_\infty}  \prod\limits_{i=1}\limits^m [w]_{A_{p_{\mathbb{X}_i}}}^{\frac{1}{p_{\mathbb{X}_i}}}+ \prod\limits_{i=1}\limits^m [w]_{A_{p_{{\mathbb{X}}_i}/r}}^{\frac{1}{q_ir}}    \right) \prod\limits_{i=1}\limits^m \left\| f_i\right\|_{\mathbb{X}_i(w)}. \\
\end{aligned}
\end{equation}\par
This completes the proof of Theorem \ref{thm1.2}.
\end{proof}

\section{ Proofs of Theorems \ref{thm1.3} and \ref{thm1.4}}
\begin{proof}[Proof of Theorem $\ref{thm1.3}$]
Note that $\mathbb{Y}:=\mathbb{X}^{\frac{1}{p}}$ is a RIBFS since $\X$ is $p$-convex and $\X _i$ is $p_i$-convex with $0<p,p_i<1, i=1,\ldots, m$, respectively. We may assume that $\mathcal{S}$ is a $\eta$-sparse family. Thus, by Lemma \ref{lem1.1}, it holds that
	\begin{equation*}
		\begin{aligned}
			\left\|{T}_{\sigma}(\vec{f})\right\|_{\X(w)}^p &\simeq \sup_{{\|g\|}_{\Y'(w)}\leq 1}\int_{\mathbb{R}^{n}}\left|{T}_{\sigma}(\vec{f})(x)\right|^pg(x)w(x)dx\\
			&\lesssim  \sup_{{\|g\|}_{\Y'(w)}\leq 1}\int_{\mathbb{R}^{n}}\Big(\sum_{Q\in \mathcal{S}}\prod_{i=1}^m\langle|f_i|\rangle_{Q} \chi_Q(x)\Big)^pg(x)w(x)dx\\
			&\lesssim  \sup_{{\|g\|}_{\Y'(w)}\leq 1}\sum_{Q\in \mathcal{S}}\prod_{i=1}^m\langle|f_i|\rangle_{Q}^p \int_{Q}g(x)w(x)dx\\
&\leq \sup_{{\|g\|}_{\Y'(w)}\leq 1} \sum_{Q\in \mathcal{S}}\left(\frac{1}{w(Q)}\int_Q\Big(\mathcal{M}(\vec{f})(x)\Big)^{\frac{p}{2}}
(M_w^{\mathcal{D}}g(x))^{\frac{1}{2}}w(x)dx\right)^2w(Q).\\
		\end{aligned}
	\end{equation*}
 Combining the Carleson embedding theorem with the above estimates gives that
\begin{equation*}\begin{aligned}
			\left\|{T}_{\sigma}(\vec{f})\right\|_{\X(w)}&\lesssim
			[w]_{A_\infty}^{\frac{1}{p}}\sup_{{\|g\|}_{\Y'(w)}\leq 1} \left(\int_{\mathbb{R}^{n}}(\mathcal{M}(\vec{f})(x))^p M_w^{\mathcal{D}}g(x)·w(x)dx\right)^{\frac{1}{p}}\\
			&\leq[w]_{A_\infty}^{\frac{1}{p}}\left\|(\mathcal{M}(\vec{f}))^p\right\|_{\Y(w)}^{\frac{1}{p}} \sup_{{\|g\|}_{\Y'(w)}\leq 1}\|M_w^{\mathcal{D}}g\|_{\Y'(w)}^{\frac{1}{p}}.\\
		\end{aligned}
	\end{equation*}
Observe that for $0<p<1<p_{\X}\leq q_{\mathbb{X}}<\infty,$
	$$p_{\mathbb{Y}^{\prime}}=\left(q_{\mathbb{Y}}\right)^{\prime}=\frac{q_{\mathbb{Y}}}{q_{\mathbb{Y}}-1}=
	\frac{q_{\mathbb{X}}}{q_{\mathbb{X}}-p}>1,$$
which indicates that $\left\|M_{w}^{\mathcal{D}}g\right\|_{\mathbb{Y}^{\prime}(w)} \lesssim\|g\|_{\mathbb{Y}^{\prime}(w)}$.\par
This fact, together with Lemma \ref{lem1.2}, further implies that	
\begin{equation*}\begin{aligned}
			\left\|{T}_{\sigma}(\vec{f})\right\|_{\X(w)}&\lesssim
			[w]_{A_\infty}^{\frac{1}{p}}\left\|\mathcal{M}(\vec{f})\right\|_{\X(w)}\\
			&\leq[w]_{A_\infty}^{\frac{1}{p}}\left\|P_m\left(Mf_1,\ldots,Mf_m \right)\right\|_{\X(w)}\\
&\lesssim [w]_{A_\infty}^{\frac{1}{p}} \prod _{i=1}^m[w]_{A_{p_{\mathbb{X}_i}}}^{\frac{1}{p_{\mathbb{X}_i}}} \left\| f_i\right\|_{\mathbb{X}_i(w)},\\
		\end{aligned}
	\end{equation*}
which finishes the proof of Theorem \ref{thm1.3}.
\end{proof}
\begin{remark}
In \cite[Theorem 6.10]{cur}, in order to ensure the extrapolation theorem holds for multilinear Calder\'{o}n-Zygmund operators, one need to assume that RIQBFS $\X$ and $\X_i$ with $i=1,\ldots,m$ enjoy the same $p$-convex property. However, in Theorem \ref{thm1.3}, we showed some weaker conditions for the extrapolation theorem such that each RIQBFS $\X_i$ may enjoys different $p_i$-convex property.
\end{remark}
\begin{proof}[Proof of Theorem $\ref{thm1.4}$]
We first note that $\mathbb{Y}=\mathbb{X}^{\frac{1}{p}}$ is a RIBFS since $\X$ is $p$-convex. By Lemma \ref{lem1.1}, it holds that
	\begin{equation}\label{eq8}
		\begin{aligned}
			{\left\|{T}_{\sigma,\Sigma \mathbf{b}}(\vec{f})\right\|}_{\X{(w)}} &\simeq \sup_{{\|h\|}_{\Y'(w)}\leq 1}\left(\int_{\mathbb{R}^{n}}\left| {T}_{\sigma,\Sigma \mathbf{b}}(\vec{f})(x)\right|^ph(x)w(x)dx\right)^{\frac{1}{p}}\\
			&\lesssim  \sup_{{\|h\|}_{\Y'(w)}\leq 1}\left(\int_{\mathbb{R}^{n}}\sum_{i=1}^{3^n}\sum_{j=1}^{m}\Big(\mathcal{A}_{\mathcal{S}_{i, b_j}}(\vec{f})(x)+\mathcal{A}_{\mathcal{S}_{i, b_j}}^{*}(\vec{f})(x)\Big)^ph(x)w(x)dx\right)^{\frac{1}{p}}\\
			&\lesssim \sup_{{\|h\|}_{\Y'(w)}\leq 1}\sum_{i=1}^{3^n}\sum_{j=1}^{m}\left(\mathcal{H}_{i,j}^1+\mathcal{H}_{i,j}^2\right)^{\frac{1}{p}},
		\end{aligned}
	\end{equation}
	where $\mathcal{H}_{i,j}^1:=\int_{\mathbb{R}^{n}}(\mathcal{A}_{\mathcal{S}_{i, b_j}}(\vec{f})(x))^ph(x)w(x)dx,$ and $\mathcal{H}_{i,j}^2:=\int_{\mathbb{R}^{n}}(\mathcal{A}_{\mathcal{S}_{i, b_j}}^*(\vec{f})(x))^ph(x)w(x)dx.$\par

Take and fix $0\leq h \in \mathbb{Y}^{\prime}(w)$ with $\|h\|_{\mathbb{Y}^{\prime}(w)} \leq 1.$ For each $1\leq i\leq 3^n, 1\leq j\leq m $ and $0<p\leq1,$ a trivial calculation shows that
\begin{equation}\label{eq6}
		\begin{aligned}
			\mathcal{H}_{i,j}^1
			&\leq\sum_{Q \in \mathcal{S}_{i}} \frac{1}{w(Q)} \int_{Q}\left|b_j(x)-b_{j,Q}\right|^p |h(x)| w(x) d x w(Q)\prod_{i=1}^m\langle|f_i|\rangle_{Q}^p\\
			&\leq 2\sum_{Q \in \mathcal{S}_{i}}\left\|\left|b_j-b_{j,Q}\right|^{p}\right\|_{\exp L^{\frac{1}{p}}(w), Q} \|h\|_{L(\log L)^p(w), Q}w(Q)\prod_{i=1}^m\langle|f_i|\rangle_{Q}^p\\
			&\lesssim[w]_{A_{\infty}}^{p}\|b_j\|_{\mathrm{BMO}}^{p}\sum_{Q \in \mathcal{S}_{i}} \|h\|_{L(\log L)^p(w), Q}w(Q)\prod_{i=1}^m\langle|f_i|\rangle_{Q}^p,\\
		\end{aligned}
	\end{equation}
where in the above inequality we have used the fact
 $\left\|\left|b-b_{Q}\right|^{p}\right\|_{\exp L^{\frac{1}{p}}(w), Q} \lesssim [w]_{A_{\infty}}^{p}\|b\|_{\mathrm{BMO}}^{p}$ (\cite[Lemma 4.7]{iba}) where
 $$
\|f\|_{\exp L^{\frac{1}{p}}(w), Q}=\inf \left\{\lambda>0: \frac{1}{w(Q)} \int_{Q} \exp \left(\frac{|f(x)|}{\lambda}\right)^{\frac{1}{p}}-1 d w<1\right\}.
$$

For any $1<\alpha<\infty$, we define $M_{\alpha(w)}^{\mathcal{D}} f$ by
$$M_{\alpha(w)}^{\mathcal{D}} f(x):=\sup _{x \in R, R \in \mathcal{D}} \left(\frac{1}{w(R)} \int_{R}|f(y)|^{\alpha} w(y) d y\right)^{\frac{1}{\alpha}}.$$
Then inequality (\ref{eq6}) implies that
\begin{equation*}
		\begin{aligned}
			\mathcal{H}_{i,j}^1
			&\lesssim [w]_{A_{\infty}}^{p}\|b_j\|_{\mathrm{BMO}}^{p}\sum_{B \in \mathcal{B}}\|h\|_{L(\log L)^p(w), B}\prod_{i=1}^m\langle|f_i|\rangle_{B}^p \sum_{Q \in S_{i}, \pi(Q)=B} w(Q)\\
			&\lesssim [w]_{A_{\infty}}^{p+1} \|b_j\|_{\mathrm{BMO}}^{p}\sum_{B \in \mathcal{B}}\|h\|_{L(\log L)^p(w), B}w(B)\prod_{i=1}^m\langle|f_i|\rangle_{B}^p \\
			&\lesssim [w]_{A_{\infty}}^{p+1} \|b_j\|_{\mathrm{BMO}}^{p} \int_{\mathbb{R}^{n}} \left(\mathcal{M} (\vec{f})(x)\right)^p M_{L (\log L)^p(w)}^{\mathcal{D}} h(x) w(x) d x\\
			&\lesssim [w]_{A_{\infty}}^{p+1} \|b_j\|_{\mathrm{BMO}}^{p} \int_{\mathbb{R}^{n}} \left(\mathcal{M} (\vec{f})(x)\right)^p M_{\alpha(w)}^{\mathcal{D}} h(x) w(x) d x,\\
		\end{aligned}
	\end{equation*}
	where $\pi(R)$ is the minimal principal cube which contains $R$ and
	$$\mathcal{B}=\cup_{k=0}^{\infty} \mathcal{B}_{k}$$
is the family of the principal cubes (\cite[P.2542]{li1})
	with $\mathcal{B}_{0}:=\{$ maximal cubes in $\mathcal{S}_j\}$ and
	$$\mathcal{B}_{k+1}:=\underset{B\in \mathcal{B}_{k}}{\cup} \operatorname{ch}_{\mathcal{B}}(B), \quad \operatorname{ch}_{\mathcal{B}}(B)=\{R \subsetneq B \text { maximal s.t. } \tau(R)>2 \tau(B)\},$$
	where $\tau(R)=\|h\|_{L(\log L)^p(w), R}\prod_{i=1}^m\langle|f_i|\rangle_{R}^p$. \par
Combining this with the generalized H\"{o}lder's inequality, we deduce that
\begin{equation*}
		\begin{aligned}
			\mathcal{H}_{i,j}^1 &\lesssim[w]_{A_{\infty}}^{p+1}\|b_j\|_{\mathrm{BMO}}^{p}\left\|\mathcal{M} (\vec{f})\right\|_{\X(w)}^p \left\|M_{w}^{\mathcal{D}}h^\alpha\right\|_{(\Y ^{\prime})^{1/\alpha}(w)}^{\frac{1}{\alpha}}.\\
		\end{aligned}
	\end{equation*}
	In order to use \cite[Lemma 3.4]{and}, we may take an $\alpha$ with $1<\alpha< (\frac{q_{\X}}{p})^{\prime}$, in this case $p_{(\Y ^{\prime})^{1/\alpha}}=\frac{p_{\Y ^{\prime}}}{\alpha}>1$.\par
Then Lemma \ref{lem1.2} yields that
\begin{equation}\label{eq7}
		\begin{aligned}
			\mathcal{H}_{i,j}^1 &\lesssim[w]_{A_{\infty}}^{p+1}\|b_j\|_{\mathrm{BMO}}^{p}\left\|h\right\|_{\Y ^{\prime}(w)}\prod_{i=1}^m[w]_{A_{p_{\X_i}}}^{\frac{p}{p_{\X_i}}}\|f_i\|_{\X_i(w)}^{p} \\
			&\leq [w]_{A_{\infty}}^{p+1}\|b_j\|_{\mathrm{BMO}}^{p}\prod_{i=1}^m[w]_{A_{p_{\X_i}}}^{\frac{p}{p_{\X_i}}}\|f_i\|_{\X_i(w)}^{p}.\\
		\end{aligned}
	\end{equation}

Now we consider $\mathcal{H}_{i,j}^2.$ Note that for $1<r<\infty,$ by the definition of $\mathcal{H}_{i,j}^2,$ we have

\begin{equation*}
		\begin{aligned}
\mathcal{H}_{i,j}^2
&\leq\sum_{Q \in \mathcal{S}_{i}}\langle\left|f_j\left(b_j-b_{j,Q}\right)\right|\rangle_{Q}^p\left(w(Q)\prod_{i\neq j}\langle |f_i| \rangle_Q^p\right) \frac{1}{w(Q)} \int_{Q} h(x) w(x) d x\\
&\leq\sum_{Q \in \mathcal{S}_{i}}w(Q)\langle\left|b_j-b_{j,Q}\right|^{r^{\prime}}\rangle_{Q}^{p /r^{\prime}} \frac{1}{w(Q)}\int_{Q} h(x) w(x) d x\langle\left|f_j\right|^{r}\rangle_{Q}^{p/r}\prod_{i \neq j}\left\langle\left|f_i\right|^r\right\rangle_Q^{p/r}\\
			&\lesssim \|b_j\|_{\mathrm{BMO}}^p \sum_{Q \in \mathcal{S}_{i}}\frac{1}{w(Q)} \int_{Q} h(x) w(x) d x w(Q) \prod_{i=1}^m\left\langle\left|f_i\right|^r\right\rangle_Q^{\frac{p}{r}}.\\
		\end{aligned}
	\end{equation*}

Consequently, applying the Carleson embedding theorem and the generalized H\"{o}lder's inequality, it gives that
\begin{equation*}
		\begin{aligned}
			\mathcal{H}_{i,j}^2&\lesssim [w]_{A_{\infty}}\|b_j\|_{\mathrm{BMO}}^p\int_{\mathbb{R}^{n}} \left(\mathcal{M}_r(\vec{f})(x)\right)^p M_{w}^{\mathcal{D}} h(x) w(x) d x\\
			&\leq [w]_{A_{\infty}}\|b_j\|_{\mathrm{BMO}}^p\left\|(\mathcal{M}_r(\vec{f}))^p\right\|_{\Y(w)} \left\|M_{w}^{\mathcal{D}} h\right\|_{\Y ^{\prime}(w)}\\
&\lesssim [w]_{A_{\infty}}\|b_j\|_{\mathrm{BMO}}^p\left\| h\right\|_{\Y^{\prime}(w)}
\prod_{i=1}^m[w]^{\frac{p}{q_ir}}_{A_{p_{\X_i}/{r}}}\| f_i\|_{{\X_i}(w)}^p.\\
		\end{aligned}
	\end{equation*}
This inequality, together with (\ref{eq7}) and (\ref{eq8}), indicates that
\begin{equation*}
\left\| T_{\sigma,\Sigma \mathbf{b}}(\vec{f})\right\|_{\mathbb{X}(w)} \lesssim  \| \vec{b}\|_{\mathrm{BMO}}[w]_{A_\infty}^{\frac{1}{p}} \left([w]_{A_\infty}^p \prod\limits_{i=1}\limits^m [w]_{A_{p_{\mathbb{X}_i}}}^{\frac{p}{p_{\mathbb{X}_i}}} + \prod\limits_{i=1}\limits^m [w]_{A_{p_{{\mathbb{X}}_i}/r}}^{\frac{p}{q_ir}} \right)^{\frac{1}{p}} \prod\limits_{i=1}\limits^m \left\| f_i\right\|_{\mathbb{X}_i(w)},
\end{equation*}
which completes the proof of Theorem \ref{thm1.4}.
\end{proof}

\section{ Proofs of Theorems \ref{thm1.5} and \ref{thm1.6}}
Before giving the proof, we need to introduce the definition  of $N$-function.
A Young function $\phi$ is called $N$-function if it satisfies the following
$$\lim _{t \rightarrow 0^{+}} \frac{\phi(t)}{t}=0 \quad \text { and } \quad \lim _{t \rightarrow \infty} \frac{\phi(t)}{t}=\infty.$$\par
Here, we list some main properties of $\phi \in \Phi$ and it's complementary function $\bar{\phi}$.
\begin{itemize}
	\item  (Young's inequality) $st \leq \phi(s)+\bar{\phi}(t), s,t \geq 0.$
	\item When $\phi$ is an $N$-function, then $\bar{\phi}$ is also an $N$-function, and satisfies the following inequalities:
	\begin{equation}\label{ie6.1}
		t \leq \phi^{-1}(t) \bar{\phi}^{-1}(t) \leq 2 t, t \geq 0;
	\end{equation}
	\begin{equation}\label{ie6.2}
		\bar{\phi}\left(\frac{\phi(t)}{t}\right) \leq \phi(t), t>0.
	\end{equation}
	\item If $\phi$ is an $N$-function, then there exists $0<\alpha<1$ such that $\phi^{\alpha}$ is quasi-convex if and only if $\bar{\phi} \in \Delta_{2},$ where $\phi^{\alpha}(t)=\phi(t)^{\alpha}.$
	\item  $\phi \in \Delta_{2}$  if and only if there exists some constant $C_1$ such that for any $\lambda \geq 2$,
	\begin{equation}\label{ie6.3}
		\phi(\lambda t) \leq 2^{C_1} \lambda^{C_1} \phi(t), t>0.
	\end{equation}
\end{itemize}\par

In addition, in order to prove Theorem \ref{thm1.5}-\ref{thm1.6}, we need the multilinear weighted dyadic Hardy-Littlewood maximal operator $\mathcal{M}_{w}^{\mathcal{D}},$ given by
 $$\mathcal{M}_{w}^{\mathcal{D}}(\vec{f})(x):=\sup _{x \in R, R \in \mathcal{D}} \prod_{i=1}^m\frac{1}{w(R)} \int_{R}|f_i(y)| w(y) d y, $$
	where $w \in A_{\infty}$ and $\mathcal{D}$ is the given dyadic grid.\par
The following estimate of modular inequality for multilinear weighted dyadic maximal operator will play a crucial role in our proof, which has some interests of its own.
\begin{lemma}\label{lem4.1}
	Let $\phi \in \Phi$ with sub-multiplicative property. If there exists $0<\alpha<1$ for which $\phi^{\alpha}$ is a quasi-convex function. Then, there exists some constant $a_{2}^{\prime}>1,$ such that for each $w \in A_{\infty},$
	\begin{equation}
		\int_{\mathbb{R}^{n}} \phi\left(\mathcal{M}_{w}^{\mathcal{D}}(\vec{f})(x)\right) w(x) d x \leq a_{2}^{\prime} \left(\prod_{i=1}^m\int_{\mathbb{R}^{n}} \phi^m\left(a_{2}^{\prime}|f_i(x)|\right) w(x) d x\right)^\frac{1}{m},
	\end{equation}
	where $\mathcal{D}$ is the given dyadic grid.
\end{lemma}

\begin{proof}
Since $\phi ^\alpha$ is a quasi-convex function, then there exists convex function $\psi$ such that
$$
\psi(t) \leq \phi^\alpha(t) \leq a_1 \psi\left(a_1 t\right), \quad t \geq 0, a_1 \geq 1.
$$
Thus for each dyadic cube  $Q\in \mathcal{D},$ using Jensen's inequality, convexity of $\psi$ and for any $\lambda \geq 1, t \geq 0, \lambda \psi(t) \leq \psi(\lambda t),$ we have
\begin{equation}\label{eq9}
\begin{aligned}
\phi^\alpha\left(\prod_{i=1}^m\frac{1}{w(Q)} \int_{Q}|f_i(y)| w(y) d y\right)
&\leq a_1 \psi\left(a_1\prod_{i=1}^m\frac{1}{w(Q)} \int_{Q}|f_i(y)| w(y) d y\right) \\
&\leq a_1 \left(\frac{1}{(w(Q))^m} \int_{Q^m}\psi (a_1\prod_{i=1}^m|f_i(y_i)|) \prod_{i=1}^m w(y_i) d y_i\right)\\
&\leq\prod_{i=1}^m\frac{1}{w(Q)} \int_{Q}\phi^\alpha (a_1^{\frac{2}{m}}|f_i(y)|) w(y) d y.
\end{aligned}
\end{equation}

Denote the vector of $\phi$ by
$\vec{f}_\phi=\left(\phi^\alpha (a_1^{\frac{2}{m}}|f_1|),\phi^\alpha (a_1^{\frac{2}{m}}|f_2|),\ldots, \phi^\alpha (a_1^{\frac{2}{m}}|f_m|)\right).$ Using this notation and (\ref{eq9}), one may obtain

$$
\phi\left(\mathcal{M}_w^{\mathcal{D}}(\vec{f})(x)\right)=\left(\phi^\alpha(\mathcal{M}_w^{\mathcal{D}}(\vec{f})(x))\right)
^{\frac{1}{\alpha}} \leq \left(\mathcal{M}_w^{\mathcal{D}}(\vec{f}_\phi)(x)\right)^{\frac{1}{\alpha}}.
$$
Therefore,
\begin{equation}\label{ie5.11}
\begin{aligned}
\int_{\mathbb{R}^n} \phi\left(\mathcal{M}_w^{\mathcal{D}}(\vec{f})(x)\right) w(x) d x &\leq \int_{\mathbb{R}^n} \left(\mathcal{M}_w^{\mathcal{D}}(\vec{f}_\phi)(x)\right)^{\frac{1}{\alpha}}w(x) d x\\
&\leq C\prod_{i=1}^m\left(\int_{\mathbb{R}^n} \phi^\alpha(a_1^{\frac{2}{m}}|f_i(x)|)^{\frac{m}{\alpha}} w(x)d x\right)^{\frac{1}{m}}\\
&\leq a_{2}^{\prime}\prod_{i=1}^m\left(\int_{\mathbb{R}^n} \phi^\alpha(a_{2}^{\prime}|f_i(x)|)^{\frac{m}{\alpha}} w(x)d x\right)^{\frac{1}{m}},
\end{aligned}
\end{equation}
where we have used the fact that $\mathcal{M}_{w}^{\mathcal{D}}(\vec{f})$ satisfy $L^{\frac{m}{\alpha}}(w) \times\cdots \times L^{\frac{m}{\alpha}}(w) \rightarrow L^{\frac{1}{\alpha}}(w)$ boundedly with  $0<\alpha<1\leq m,$ which derived directly from inequality $\mathcal{M}_{w}^{\mathcal{D}}(\vec{f})(x)\leq \prod _{i=1}^mM_{w}^{\mathcal{D}}(f_i)(x),$ the boundedness of $M_{w}^{\mathcal{D}}$ and the H\"{o}lder's inequality.
\end{proof}
\begin{remark}\label{rem1}
From the proof, it is clear that Lemma \ref{lem4.1} holds with $m=1$ without assume the sub-multiplication condition of $\phi$.
\end{remark}\par
We also need the following version of the modular inequalities for the multilinear maximal operator.
\begin{lemma}\label{lem4.2}
	Let $\phi \in \Phi$ be a quasi-convex function with sub-multiplicative property. For each~$1\leq r < \infty,$ if $r<i_{\phi}<\infty$ then there exists constant $a_3$ such that for
every~$1<q<\frac{i_{\phi}}{r}$ and ~$w\in A_q ,$
	\begin{equation}
		\int_{\mathbb{R}^n} \phi\left(\mathcal{M}_r(\vec{f})(x)\right) w(x) d x \leq a_3\left(\prod_{i=1}^m\int_{\mathbb{R}^n} \phi^m\left(a_3[w]_{A_q}^{\frac{1}{qr}}\left|f_i(x)\right|\right) w(x) d x\right)^{\frac{1}{m}}.
	\end{equation}
\end{lemma}
\begin{proof}
Observe that for any dyadic cube $Q \in \mathcal{D}$, for each $x\in Q$, it holds that
\begin{equation*}\begin{aligned}
			\prod_{i=1}^m\frac{1}{|Q|} \int_{Q}|f_i|^{r}(y)dy
			&\leq\prod_{i=1}^m\left(\frac{1}{|Q|} \int_{Q}|f_i|^{qr}(y) w(y) dy\right)^{\frac{1}{q}}\left(\frac{1}{|Q|} \int_{Q} w^{-\frac{q'}{q}}(y)dy\right)^{\frac{1}{q'}} \\
			&\leq\prod_{i=1}^m\left[\left(\frac{1}{|Q|} \int_{Q} w(y)dy\right)\left(\frac{1}{|Q|} \int_{Q} w^{-\frac{1}{q-1}}(y)dy\right)^{q-1}\right]^{\frac{1}{q}}\\
			&\quad \times \left(\frac{1}{w(Q)} \int_{Q}|f_i|^{qr}(y) w(y)dy\right)^{\frac{1}{q}}\\
			&\leq [w]^{\frac{m}{q}}_{A_q}\left( \mathcal{M}_w^{\mathcal{D}}((\vec{f})^{qr})(x)\right)^{\frac{1}{q}},
	\end{aligned}\end{equation*}
where $(\vec{f})^s$ denotes $(f_1^s,\ldots,f_m^s).$\par
Therefore
$$\mathcal{M}_r(\vec{f})(x)\leq [w]^{\frac{m}{rq}}_{A_q}\left( M_w^{\mathcal{D}}((\vec{f})^{qr})(x)\right)^{\frac{1}{rq}}.$$
For $s,t>0$, set $\phi_t(s)=\phi(s^\frac{1}{t}).$ By the properties of $\phi$, we have
\begin{equation*}\begin{aligned}
	\phi(\mathcal{M}_r(\vec{f})(x))&\leq\phi\left([w]^{\frac{m}{rq}}_{A_q}\left( \mathcal{M}_w^{\mathcal{D}}((\vec{f})^{qr})(x)\right)^{\frac{1}{rq}}\right)\\
&=\phi_{qr}\left( \mathcal{M}_{w}^{\mathcal{D}}\left([w]_{A_{q}}(\vec{f})^{qr}\right)(x)\right).
\end{aligned}\end{equation*}

Since $1<rq<i_\phi$,  then by \cite[Lemma 5.2]{cur}, there exists a constant $0<\alpha<1$, such that $\phi_{qr}(s)^\alpha=\phi\left(s^{\frac{1}{qr}}\right)^\alpha$ is a quasi-convex function. This fact, together with Lemma \ref{lem4.1} further implies that
\begin{equation*}\begin{aligned}
			\int_{\mathbb{R}^{n}} \phi(\mathcal{M}_{r}( \vec{f})(x)) w(x) dx &\leq  \int_{\mathbb{R}^{n}} \phi_{qr}\left( M_{w}^{\mathcal{D}}\left([w]_{A_{q}}(\vec{f})^{qr}\right)(x)\right) w(x)dx \\
			&\leq a_2^{\prime}\left(\prod_{i=1}^m\int_{\mathbb{R}^{n}} \phi_{qr}^m(a_{2}^{\prime}[w]_{A_{q}}|f_i(x)|^{qr}) w(x)dx\right)^{\frac{1}{m}}\\
			&\leq a_3\left(\prod_{i=1}^m\int_{\mathbb{R}^n} \phi^m\left(a_3[w]_{A_q}^{\frac{1}{qr}}\left|f_i(x)\right|\right) w(x) d x\right)^{\frac{1}{m}}.\\
		\end{aligned}
	\end{equation*}
This finishes the proof of Lemma \ref{lem4.2}.
\end{proof}
For the sake of convenience, we will make a summary of the constants in the inequality which we need to use later.
\begin{equation}\label{ie5.1}
	constant~ a_1: \quad 	\psi(t)\leq \phi(t)\leq a_1\psi(a_1t), t>0.
	\end{equation}
\begin{equation}\label{ie5.2}
	constant~ a_2^{\prime}: \quad 	\int_{\mathbb{R}^{n}} \phi\left(M_{w}^{\mathcal{D}}(\vec{f})(x)\right) w(x) d x \leq a_{2}^{\prime} \left(\prod_{i=1}^m\int_{\mathbb{R}^{n}} \phi^m\left(a_{2}^{\prime}|f_i(x)|\right) w(x) d x\right)^\frac{1}{m}
	\end{equation}
\begin{equation}\label{ie5.3}
	constant~ C_1: \quad 	\phi(\lambda t) \leq 2^{C_1} \lambda^{C_1} \phi(t), \quad \lambda \geq 2, t>0.
	\end{equation}
The following lemma gives an estimate of the modular inequality for the sparse operator.
\begin{lemma}\label{lem4.3}
	Let $\phi$ be a $N$-function with sub-multiplicative property. Then there exists $0<\alpha <1,$ such that for any dyadic lattice $\mathcal{D},$ sparse family $\mathcal{S}\subseteq \mathcal{D}$ and $w\in A_{\infty},$
	\begin{equation*}
		\int_{\mathbb{R}^{n}} \phi\left( \mathcal{A}_{\mathcal{S}}(\vec{f})(x)\right) w(x) d x \lesssim [w]_{A_{\infty}}^{1+\alpha C_1} \int_{\mathbb{R}^{n}} \phi\left(\mathcal{M}(\vec{f})(x)\right) w(x) d x,
	\end{equation*}
	where $C_1$ is the same constant defined in $(\ref{ie5.3})$.
\end{lemma}
\begin{proof}
Since $\phi$ satisfies sub-multiplicative condition, that is, for each $t_1,t_2\geq 0,$ $\phi(t_1t_2)\leq \phi(t_1)\phi(t_2).$ Fixed $t_1=2,$  for $t\geq 0$ we have
$\phi(2t)\leq \phi(2)\phi(t) $ which implies that $\phi \in \Delta _2.$
Therefore, there exists some $0<\alpha < 1,$ such that $\bar{\phi}^{\alpha}$ is quasi-convex, i.e., there exists some convex function $\psi$ and $a_1> 1$ such that
	\begin{equation}\label{ie5.4}
		\psi(t)\leq \bar\phi^\alpha(t)\leq a_1\psi(a_1t), t>0.
	\end{equation}
Observe that, we can always assume $a_1\geq 2$ by iteration. We also assume that $\mathcal{S}$ is $\eta$-sparse family where $\eta$ is an absolute constant only depending on $n.$\par
Note that $\phi(\mathcal{A}_{\mathcal{S}}(\vec{f})(x))=0$ when $\mathcal{A}_{\mathcal{S}}(\vec{f})(x)=0$ since $\phi$ is a $N$-function. Now, we define the function $h$ having the following form.
\begin{equation*}
h(x)= \begin{cases}0, &  \mathcal{A}_{\mathcal{S}}(\vec{f})(x)=0, \\ \frac{\phi(\mathcal{A}_{\mathcal{S}}(\vec{f})(x))}{\mathcal{A}_{\mathcal{S}}(\vec{f})(x)}, & otherwise.\end{cases}
\end{equation*}
Let $E_1=\{x\in \mathbb{R}^n: \mathcal{A}_{\mathcal{S}}(\vec{f})(x)=0\}.$ By the definition of $\mathcal{A}_{\mathcal{S}}(\vec{f})$ and the embedding theorem, it shows that
\begin{equation*}\begin{aligned}
			\int_{\mathbb{R}^{n}}\phi(\mathcal{A}_{\mathcal{S}}(\vec{f})(x))w(x)dx &=\int_{\mathbb{R}^{n}\setminus E_1}\phi(\mathcal{A}_{\mathcal{S}}(\vec{f})(x))w(x)dx\\
			&\leq \frac{4}{\eta}[w]_{A_\infty}\int_{\mathbb{R}^{n}}\mathcal{M}(\vec{f})(x)
(M_w^{\mathcal{D}}h(x))w(x)dx.\\
	\end{aligned}\end{equation*}

We now take an $\varepsilon$ with the property that
	\begin{equation*}
		0<\varepsilon=\min\left\{\frac{1}{2},~\frac{1}{a_1a_2^{\prime}}, ~(\frac{\eta}{8[w]_{A_\infty}a_2^{\prime}})^\alpha\cdot\frac{1}{a_2^{\prime}a_1^2}\right\}.
	\end{equation*}
Using the Young's inequality of $\phi,$ it is easy to verify that
\begin{equation*}\begin{aligned}
			\int_{\mathbb{R}^{n}}\phi(\mathcal{A}_{\mathcal{S}}(\vec{f})(x))w(x)dx &\leq \frac{4}{\eta}[w]_{A_\infty}\int_{\mathbb{R}^{n}}\frac{\mathcal{M}(\vec{f})(x)}{\varepsilon}
(M_w^{\mathcal{D}}(\varepsilon h(x))w(x)dx\\
&\leq\frac{4}{\eta}[w]_{A_\infty}\int_{\mathbb{R}^{n}}\left[\phi(\frac{\mathcal{M}(\vec{f})(x)}{\varepsilon})+\bar\phi \left(M_w^{\mathcal{D}}(\varepsilon h(x))\right)\right]w(x)dx\\
			&\leq\frac{4}{\eta}[w]_{A_\infty}2^{C_1}\varepsilon^{-C_1}\int_{\mathbb{R}^{n}}\phi(\mathcal{M}(\vec{f})(x))w(x)dx\\
&\quad+\frac{4}{\eta}[w]_{A_\infty} a_2^{\prime}\int_{\mathbb{R}^{n}}\bar\phi(a_2^{\prime}\varepsilon h(x))w(x)dx,
\end{aligned}\end{equation*}
where the last inequality follows from the fact that $\phi \in \Delta_2$ with $\varepsilon\leq \frac{1}{2}$ and Remark \ref{rem1}.\par
Now we turn our attention to $\int_{\mathbb{R}^{n}}\bar\phi(a_2^{\prime}\varepsilon h(x))w(x)dx.$
Observe that $\psi $ is a convex function, then for	any $0\leq \lambda\leq1, t\geq0,$ $\psi(\lambda t)\leq\lambda\psi(t).$\par
	Therefore, by $(\ref{ie5.4})$, we get
	$$\bar\phi^{\alpha}(a_2^{\prime}\varepsilon h(x))\leq a_1^2a_2^{\prime}\varepsilon\psi(h(x))\leq a_1^2a_2^{\prime}\varepsilon \bar\phi^\alpha(h(x)),$$
where in the above inequality we have used $a_1a_2^{\prime}\varepsilon\leq 1.$\par
This means that
\begin{equation}\label{ie5.5}
\bar\phi(a_2^{\prime}\varepsilon h(x))\leq(a_1^2a_2^{\prime}\varepsilon)^{\frac{1}{\alpha}}\bar\phi(h(x)).
\end{equation}
By the definition of $h$ and $\bar{\phi}\left(\frac{\phi(t)}{t}\right) \leq \phi(t),t>0,$ we see that if $\mathcal{A}_{\mathcal{S}}(\vec{f})(x)\neq0,$ then
$$\bar{\phi}(h(x))= \bar{\phi}\left(\frac{\phi(\mathcal{A}_{\mathcal{S}}(\vec{f})(x))}{\mathcal{A}_{\mathcal{S}}(\vec{f})(x)}\right)
\leq \phi (\mathcal{A}_{\mathcal{S}}(\vec{f})(x)).$$
If $\mathcal{A}_{\mathcal{S}}(\vec{f})(x)=0,$ by the definitions of $\bar{\phi},$
it is easy to check that
$$\bar{\phi}(h(x))= \bar{\phi}\left(0\right)=0
\leq \phi (\mathcal{A}_{\mathcal{S}}(\vec{f})(x)).$$
Combining the above two cases and integrating (\ref{ie5.5}), we deduce that
\begin{equation*}
\begin{aligned}
	\int_{\mathbb{R}^{n}}\phi(\mathcal{A}_{\mathcal{S}}(\vec{f})(x))w(x)dx &\leq \frac{4}{\eta}[w]_{A_\infty}2^{C_1}\varepsilon^{-C_1}\int_{\mathbb{R}^{n}}\phi(\mathcal{M}(\vec{f})(x))w(x)dx\\
&\quad+\frac{4}{\eta}[w]_{A_\infty} a_2^{\prime}(a_1^2a_2^{\prime}\varepsilon)^{\frac{1}{\alpha}} \int_{\mathbb{R}^{n}}\bar\phi(h(x))w(x)dx\\
			&\leq\frac{4}{\eta}[w]_{A_\infty}2^{C_1}\varepsilon^{-C_1}\int_{\mathbb{R}^{n}}\phi(\mathcal{M}(\vec{f})(x))w(x)dx\\
&\quad+\frac{1}{2} \int_{\mathbb{R}^{n}}\phi (\mathcal{A}_{\mathcal{S}}(\vec{f})(x))w(x)dx.\\
\end{aligned}
\end{equation*}
Moreover, a direct calculation shows that $[w]_{A_\infty}\varepsilon^{-C_1}\lesssim [w]_{A_\infty}^{1+\alpha C_1},$ which enable us to conclude that
\begin{equation*}
		\int_{\mathbb{R}^{n}} \phi\left( \mathcal{A}_{\mathcal{S}}(\vec{f})(x)\right) w(x) d x \lesssim [w]_{A_{\infty}}^{1+\alpha C_1} \int_{\mathbb{R}^{n}} \phi\left(\mathcal{M}(\vec{f})(x)\right) w(x) d x.
	\end{equation*}
This completes the proof of Lemma \ref{lem4.3}.
\end{proof}
Now we begin the proof of the main theorems.
\begin{proof}[Proof of Theorem $\ref{thm1.5}$]
We divide the proof into two parts. \\
\textbf{Case 1}: We start with the case $i_\phi>1.$ By Lemma \ref{lem1.1} and the sub-multiplicative property of $\phi,$ we have
\begin{equation}\label{ie5.13}
\begin{aligned}
\int_{\mathbb{R}^n}\phi (|T_{\sigma}(\vec{f})(x)|)w(x)dx &\lesssim \int_{\mathbb{R}^n}\phi(\mathcal{A}_\mathcal{S} (\vec{f}) (x))w(x)dx  \\
&\lesssim [w]_{A_\infty}^{1+\alpha C_1}\int_{\mathbb{R}^n}\phi(\mathcal{M} (\vec{f}) (x))w(x)dx \\
&\lesssim[w]_{A_\infty}^{1+\alpha C_1}\left(\prod_{i=1}^m \int_{\mathbb{R}^n}\phi^m (a_3[w]_{A_q}^{\frac{1}{q}}|f_i(x)|)w(x)dx\right)^{\frac{1}{m}},
\end{aligned}
\end{equation}
where we have applied Lemma \ref{lem4.3} and Lemma \ref{lem4.2} with $r=1.$\par
For the sake of accuracy, we discuss two scenarios.\par
If $a_3[w]_{A_q}^{\frac{1}{q}} \leq 2,$ by the sub-multiplicative property and the monotonicity of $\phi$,
$$\int_{\mathbb{R}^n}\phi (|T_{\sigma}(\vec{f})(x)|)w(x)dx  \lesssim [w]_{A_\infty}^{1+\alpha C_1}\left(\prod_{i=1}^m \int_{\mathbb{R}^n}\phi^m (|f_i(x)|)w(x)dx\right)^{\frac{1}{m}}.$$
For the case $a_3[w]_{A_q}^{\frac{1}{q}} > 2,$ by (\ref{ie5.3}),  we obtain that
$$
\int_{\mathbb{R}^n}\phi (|T_{\sigma}(\vec{f})(x)|)w(x)dx  \lesssim [w]_{A_\infty}^{1+\alpha C_1}[w]_{A_q}^{\frac{C_1}{q}m}\left(\prod_{i=1}^m \int_{\mathbb{R}^n}\phi^m ( |f_i(x)|)w(x)dx\right)^{\frac{1}{m}}.
$$\par

\textbf{Case 2}: Now, we consider the endpoint case $i_\phi=1.$ In this case, by (\ref{ie5.13}), applying the extrapolation theorem \cite[Theorem 3.1]{cur} to $(\phi (|T_{\sigma}(\vec{f})(x)|),\phi (\mathcal{M}(\vec{f})(x))$, we have
$${||\phi (|T_{\sigma}(\vec{f})(x)|)||}_{\X(w)} \lesssim {||\phi (\mathcal{M}(\vec{f}))||}_{\X(w)}, $$
holds for RIQBFS $\X$ with $q_{\X}<\infty.$\par
Let $\X=L^{\frac{1}{m},\infty }(\mathbb{R}^n),$ we observe that
$$\sup_{\lambda>0} \phi(\lambda)w\left(\{ x\in\mathbb{R}^n:|T_b(\vec{f})(x)|>\lambda\}\right)^m  \lesssim \sup_{\lambda>0} \phi(\lambda)w\left(\{ x\in\mathbb{R}^n:\mathcal{M}(\vec{f})(x)>\lambda\}\right)^m.$$
For any dyadic cube $Q,$ and $x\in Q,$ it is easy to see
\begin{equation}
\begin{aligned}
\prod_{i=1}^m  \frac{1}{|Q|} \int_Q |f_i(y)|dy
& \leq[w]_{A_1}^m \prod_{i=1}^m \frac{1}{w(Q)} \int_Q |f_i(y)|w(y)dy  \\
& \leq [w]_{A_1}^m \mathcal{M}_w^{\mathcal{D}}(\vec{f})(x),   \\
\end{aligned}
\end{equation}
which further implies that
$\phi(\mathcal{M}(\vec{f})(x)) \leq \phi ([w]_{A_1}^m \mathcal{M}_w^{\mathcal{D}}(\vec{f})(x) ).$ \par
This fact together with Lemma \ref{lem4.1} gives that
$$\phi(\mathcal{M}(\vec{f})(x)) \leq \mathcal{M}_w^{\mathcal{D}}(\phi([w]_{A_1}f_1),\cdots,\phi([w]_{A_1}f_m))(x)=:\mathcal{M}_w^{\mathcal{D}}(\vec{f_\phi})(x).$$

Using the method of Theorem 1.5 in \cite{li2}, we can obtain that $\mathcal{M}_w^{\mathcal{D}}(\vec{f})$ is bounded from
$L^1(w)\times\cdots\times L^1(w)$ into $L^{\frac{1}{m},\infty }(w).$\par
Then, for any $\lambda >0,$ it is easy to verify that
\begin{equation*}
\begin{aligned}
w(\{x \in \mathbb{R}^n:\mathcal{M}(\vec{f})(x)>\lambda\})
&=w(\{x \in \mathbb{R}^n:\phi(\mathcal{M}(\vec{f})(x))>\phi(\lambda)\})\\
& \lesssim w(\{x \in \mathbb{R}^n:\mathcal{M}_w^{\mathcal{D}}(\vec{f}_\phi)(x)>\phi(\lambda)\})  \\
& \lesssim \left(\frac{1}{\phi(\lambda)} \prod_{i=1}^m {||f_{i,\phi}||}_{L^1(w)}\right)^{\frac{1}{m}}. \\
\end{aligned}
\end{equation*}
Therefore
$$\sup_{\lambda>0} \phi(\lambda)w(\{x \in \mathbb{R}^n:\mathcal{M}(\vec{f})(x)>\lambda\})^m  \lesssim
\prod_{i=1}^m \int_{\mathbb{R}^n} \phi([w]_{A_1}f_i(x) ) w(x)dx,$$
which finish the proof of Theorem \ref{thm1.5} with the case $i_\phi=1,$ and also completes the proof of Theorem \ref{thm1.5}.
\end{proof}
\begin{remark}
The above proof shows that if $\phi$ satisfies $\phi(\prod_{i=1}^{m}t_i)\leq \prod_{i=1}^{m}\phi _i(t_i)$ with $N$-functions $\phi_1,\ldots,\phi_m,$ then Theorem \ref{thm1.5} has the more general form as follows.
\begin{equation*}
\int_{\mathbb{R}^n}\phi(|T_{\sigma}(\vec{f})(x)|)w(x)dx\leq C(\phi,w)
 \left(\prod\limits_{i=1}^{m}\int_{\mathbb{R}^n}\phi_i^m(|f_i(x)|)w(x)dx \right)^{\frac{1}{m}}.
\end{equation*}
\end{remark}
\begin{proof}[Proof of Theorem $\ref{thm1.6}$]
By using the sparse domination results of Lemma \ref{lem1.1}, we see that
\begin{equation*}
\begin{aligned}
			\left|T_{\sigma, \Sigma \mathbf{b}}(\vec{f})(x)\right|
&\leq\max\{2, 3^n\cdot Cm\}\sum_{i=1}^{3^n} \sum_{j=1}^m\left(\frac{\mathcal{A}_{\mathcal{S}_i, b_j}(\vec{f})(x)+\mathcal{A}_{\mathcal{S}_i, b_j}^*(\vec{f})(x)}{3^nm}\right).\\
\end{aligned}
\end{equation*}
Using condition $\phi \in \Delta_2$ and the convexity of $\phi,$ it holds that

\begin{equation*}
\begin{aligned}
			\phi\left(|T_{\sigma, \Sigma \mathbf{b}}(\vec{f})(x)|\right)
&\leq \frac{C}{3^n2m}\sum_{i=1}^{3^n} \sum_{j=1}^m \left[\phi\left(\mathcal{A}_{\mathcal{S}_i, b_j}(\vec{f})(x)\right)+\phi\left(\mathcal{A}_{\mathcal{S}_i, b_j}^*(\vec{f})(x)\right)\right].\\
\end{aligned}
\end{equation*}
From this inequality, it follows that
\begin{equation}\label{ie5.7}
\begin{aligned}
		\int_{\mathbb{R}^{n}} \phi\left(|T_{\sigma, \Sigma \mathbf{b}}(\vec{f})(x)|\right) w(x) d x \lesssim  \sum_{i=1}^{3^n} \sum_{j=1}^m &\left(\int_{\mathbb{R}^{n}} \phi(\mathcal{A}_{\mathcal{S}_i, b_j}(\vec{f})(x)) w(x) d x\right. \\
&+ \left.\int_{\mathbb{R}^{n}}\phi\left(\mathcal{A}_{\mathcal{S}_i, b_j}^*(\vec{f})(x)\right)w(x)dx \right).\\
\end{aligned}
	\end{equation}\par

Consider now the contribution of $\int_{\mathbb{R}^{n}} \phi(\mathcal{A}_{\mathcal{S}_i, b_j}(\vec{f})(x)) w(x) d x$ for $1\leq i \leq3^n, 1\leq j\leq m.$ Using the Carleson embedding theorem, we apply (\ref{eq10}) in Theorem \ref{thm1.2} to obtain

\begin{equation}\label{ie5.8}
\begin{aligned}
	\int_{\mathbb{R}^{n}} \phi(\mathcal{A}_{\mathcal{S}_i, b_j}(\vec{f})(x)) w(x) d x& =\sum_{Q \in \mathcal{S}_{i}}\int_{Q}\left|b_j(x)-b_{j,Q}\right|\frac{\phi(\mathcal{A}_{\mathcal{S}_{i}, b_j} (\vec{f})(x))}{\mathcal{A}_{\mathcal{S}_{i}, b_j} (\vec{f})(x)}w(x) d x \left(\prod_{i=1}^m\langle|f_i|\rangle_{Q}\right)\\
			&\leq 2\sum_{Q \in \mathcal{S}_{i}} \left\|b_j-b_{j,Q}\right\|_{\exp L(w), Q}\|\frac{\phi(\mathcal{A}_{\mathcal{S}_{i}, b_j} (\vec{f})(x))}{\mathcal{A}_{\mathcal{S}_{i}, b_j} (\vec{f})(x)} \|_{L(\log L)(w), Q}\\
&\quad \times w(Q)\left(\prod_{i=1}^m\langle|f_i|\rangle_{Q}\right)\\
			&\leq C^{\prime}[w]_{A_{\infty}}^2\|b_j\|_{\mathrm{BMO}} \\
&\quad \times \int_{\mathbb{R}^{n}} \mathcal{M}(\vec{f})(x) (M_{w}^{\mathcal{D}})^{2}\left(\frac{\phi(\mathcal{A}_{\mathcal{S}_{i}, b_j} (\vec{f})(x))}{\mathcal{A}_{\mathcal{S}_{i}, b_j} (\vec{f})(x)}\right)(x) w(x) d x,\\
		\end{aligned}
	\end{equation}
where $C^{\prime}>0$ is an absolute constant which is independent of $w$.\par
For any $1\leq j\leq m,$ we take some $\varepsilon_j$ in the following form:
\begin{equation*}
		0<\varepsilon_j=\min\left\{\frac{1}{2},~\frac{1}{a_1a_2^{\prime}}, ~\left(\frac{1}{2C^{\prime}[w]_{A_\infty}^2 (a_2^{\prime})^2\|b_j\|_{\mathrm{BMO}}}\right)^\alpha\cdot\frac{1}{(a_2^{\prime})^2a_1^2}\right\}.
	\end{equation*}
Then the Young's inequality gives that
\begin{equation}\label{ie5.9}
\begin{aligned}
	\int_{\mathbb{R}^{n}} \mathcal{M}(\vec{f})(x) (M_{w}^{\mathcal{D}})^{2}&\left(\frac{\phi(\mathcal{A}_{\mathcal{S}_{i}, b_j} (\vec{f})(x))}{\mathcal{A}_{\mathcal{S}_{i}, b_j} (\vec{f})(x)}\right)(x) w(x) d x\\
& \leq2^{C_1}\varepsilon_j^{-C_1}\int_{\mathbb{R}^{n}}\phi(\mathcal{M}(\vec{f})(x))w(x)dx\\
&\quad+ a_2^{\prime}(a_1^2(a_2^{\prime})^2\varepsilon_j)^{\frac{1}{\alpha}} \int_{\mathbb{R}^{n}}\bar\phi\left(\frac{\phi(\mathcal{A}_{\mathcal{S}_{i}, b_j} (\vec{f})(x))}{\mathcal{A}_{\mathcal{S}_{i}, b_j} (\vec{f})(x)}\right)w(x)dx\\
			&\leq2^{C_1}\varepsilon_j^{-C_1}\int_{\mathbb{R}^{n}}\phi(\mathcal{M}(\vec{f})(x))w(x)dx\\
&\quad+ \frac{1}{2C^{\prime}[w]_{A_{\infty}}^2\|b_j\|_{\mathrm{BMO}}} \int_{\mathbb{R}^{n}} \phi(\mathcal{A}_{\mathcal{S}_i, b_j}(\vec{f})(x)) w(x) d x,\\
		\end{aligned}
	\end{equation}
where in the first inequality we have used Lemma \ref{lem4.1} twice and obtained
\begin{equation*}
		\int_{\mathbb{R}^{n}}\bar\phi((M_w^{\mathcal{D}})^2f(x))w(x)dx\leq (a_2^{\prime})^2 \int_{\mathbb{R}^{n}}\bar\phi\left((a_2^{\prime})^2|f(x)|\right)w(x)dx.
	\end{equation*}
Applying the estimates (\ref{ie5.8}) and (\ref{ie5.9}), we may obtain
\begin{equation}\label{ie5.10}
\int_{\mathbb{R}^{n}} \phi(\mathcal{A}_{\mathcal{S}_i, b_j}(\vec{f})(x)) w(x) d x\lesssim [w]_{A_\infty}^{2+2\alpha C_1}\|b_j\|_{\mathrm{BMO}}^{1+\alpha C_1}\int_{\mathbb{R}^{n}} \phi\left(\mathcal{M}(\vec{f})(x)\right) w(x) d x.
\end{equation}\par

Now we turn to the proof of $\int_{\mathbb{R}^{n}} \phi(\mathcal{A}_{\mathcal{S}_i, b_j}^*(\vec{f})(x)) w(x) d x$ for $1\leq i \leq3^n, 1\leq j\leq m.$
The same reasoning as what we have dealt with $\mathcal{I}_{i,j}^2,$ together with the H\"{o}lder inequality,  yields that
\begin{equation*}
		\begin{aligned}
			\int_{\mathbb{R}^{n}} \phi\left(\mathcal{A}_{\mathcal{S}_{i}, b_j}^{*}(\vec{ f})(x)\right) w(x) d x
			&\leq\sum_{Q \in \mathcal{S}_{i}}\langle\left|b_j-b_{j,Q}\right|^{r^{\prime}}\rangle_{Q}^{1 /r^{\prime}}\left(\prod_{i=1}^m\left\langle\left|f_i\right|^r\right\rangle_{Q}^{1/r}\right)\\
& \qquad \times \int_Q\frac{\phi(\mathcal{A}_{\mathcal{S}_{i}, b_j}^{*}(\vec{ f})(x))}{\mathcal{A}_{\mathcal{S}_{i}, b_j}^{*}(\vec{ f})(x)} w(x) d x \\
			&\lesssim  \|b_j\|_{\mathrm{BMO}}\sum_{Q \in \mathcal{S}_{i}} \int_Q\frac{\phi(\mathcal{A}_{\mathcal{S}_{i}, b_j}^{*}(\vec{ f})(x))}{\mathcal{A}_{\mathcal{S}_{i}, b_j}^{*}(\vec{ f})(x)} w(x) d x\prod_{i=1}^m\left\langle\left|f_i\right|^r\right\rangle_{Q}^{1/r}\\
			&\lesssim  \|b_j\|_{\mathrm{BMO}}[w]_{A_{\infty}} \\
& \quad \times \int_{\mathbb{R}^{n}} \mathcal{M}_{r}(\vec{f})(x) M_{w}^{\mathcal{D}}\left(\frac{\phi(\mathcal{A}_{\mathcal{S}_{i}, b_j}^{*}(\vec{ f}))}{\mathcal{A}_{\mathcal{S}_{i}, b_j}^{*}(\vec{ f})}\right)(x) w(x) d x.\\
		\end{aligned}
	\end{equation*}
Then, some elementary calculations give that
\begin{equation}\label{ie5.14}
\int_{\mathbb{R}^{n}} \phi\left(\mathcal{A}_{\mathcal{S}_{i}, b_j}^{*}(\vec{ f})(x)\right) w(x) d x \lesssim
 [w]_{A_\infty}^{1+2\alpha C_1}\|b_j\|_{\mathrm{BMO}}^{1+\alpha C_1}\int_{\mathbb{R}^{n}} \mathcal{M}_{r}(\vec{f})(x)w(x) d x .
 \end{equation}
We are now in the position to finish our proof. First, by the estimates (\ref{ie5.7}), (\ref{ie5.10}) and (\ref{ie5.14}), we get
\begin{equation*}
\begin{aligned}
		\int_{\mathbb{R}^{n}} \phi\left(|T_{\sigma, \Sigma \mathbf{b}}(\vec{f})(x)|\right) w(x) d x \lesssim  [w]_{A_\infty}^{1+2\alpha C_1}\|\vec{b}\|_{\mathrm{BMO}}^{1+\alpha C_1} &\left([w]_{A_\infty} \int_{\mathbb{R}^{n}} \phi\left(\mathcal{M}(\vec{f})(x)\right) w(x) d x\right. \\
&+ \left.\int_{\mathbb{R}^{n}}\mathcal{ M}_{r}(\vec{f})(x)w(x) d x  \right).\\
\end{aligned}
\end{equation*}\par
Then, for any $r>1$ and $1<q<\frac{i_{\phi}}{r},$ there exist $a_3$ and $a_3^{\prime}$ satisfying Lemma \ref{lem4.2}, such that
	\begin{equation*}
		\int_{\mathbb{R}^n} \phi\left(M_r(\vec{f})(x)\right) w(x) d x \leq a_3^{\prime}\left(\prod_{i=1}^m\int_{\mathbb{R}^n} \phi^m\left(a_3^{\prime}[w]_{A_q}^{\frac{1}{qr}}\left|f_i(x)\right|\right) w(x) d x\right)^{\frac{1}{m}};
	\end{equation*}
 \begin{equation*}
		\int_{\mathbb{R}^n} \phi\left(\mathcal{M}(\vec{f})(x)\right) w(x) d x \leq a_3\left(\prod_{i=1}^m\int_{\mathbb{R}^n} \phi^m\left(a_3[w]_{A_q}^{\frac{1}{q}}\left|f_i(x)\right|\right) w(x) d x\right)^{\frac{1}{m}}.
	\end{equation*}
To finish the proof, we only need to consider four cases for modular inequalities.\par

Case \uppercase\expandafter{\romannumeral1}: $a_3[w]_{A_{q}}^{\frac{1}{q}}<2,$
	$a_3^{\prime}[w]_{A_q}^{\frac{1}{qr}}<2$.\par
	In this case, we have
	\begin{equation*}
\begin{aligned}
		\int_{\mathbb{R}^{n}} \phi\left(|T_{\sigma, \Sigma \mathbf{b}}(\vec{f})(x)|\right) w(x) d x
\lesssim & [w]_{A_\infty}^{1+2\alpha C_1}\|\vec{b}\|_{\mathrm{BMO}}^{1+\alpha C_1}([w]_{A_\infty}+1) \\ &\times\left(\prod\limits_{i=1}^{m}\int_{\mathbb{R}^n}\phi^m(2|f_i(x)|)w(x)dx \right)^{\frac{1}{m}}. \\
\end{aligned}
\end{equation*}\par
	Case \uppercase\expandafter{\romannumeral2}: $a_3[w]_{A_{q}}^{\frac{1}{q}}\geq 2,$
	$a_3^{\prime}[w]_{A_q}^{\frac{1}{qr}}<2$.\par
	For this case, by $\Delta_2$ condition, we obtain
	\begin{equation*}
\begin{aligned}
		\int_{\mathbb{R}^{n}} \phi\left(|T_{\sigma, \Sigma \mathbf{b}}(\vec{f})(x)|\right) w(x) d x
\lesssim & [w]_{A_\infty}^{1+2\alpha C_1}\|\vec{b}\|_{\mathrm{BMO}}^{1+\alpha C_1}([w]_{A_\infty}[w]_{A_q}^{\frac{C_1m}{q}}+1) \\ &\times\left(\prod\limits_{i=1}^{m}\int_{\mathbb{R}^n}\phi^m(2|f_i(x)|)w(x)dx \right)^{\frac{1}{m}}. \\
\end{aligned}
\end{equation*}\par
	Case \uppercase\expandafter{\romannumeral3}: $a_3[w]_{A_{q}}^{\frac{1}{q}}< 2,$
	$a_3^{\prime}[w]_{A_q}^{\frac{1}{qr}}\geq 2$.\par
	To see this case, using $\Delta_2$ condition again, we get
	\begin{equation*}
\begin{aligned}
		\int_{\mathbb{R}^{n}} \phi\left(|T_{\sigma, \Sigma \mathbf{b}}(\vec{f})(x)|\right) w(x) d x
\lesssim & [w]_{A_\infty}^{1+2\alpha C_1}\|\vec{b}\|_{\mathrm{BMO}}^{1+\alpha C_1}([w]_{A_\infty}+[w]_{A_q}^{\frac{C_1m}{qr}}) \\ &\times\left(\prod\limits_{i=1}^{m}\int_{\mathbb{R}^n}\phi^m(2|f_i(x)|)w(x)dx \right)^{\frac{1}{m}}. \\
\end{aligned}
\end{equation*}\par
	Case \uppercase\expandafter{\romannumeral4}: $a_3[w]_{A_{q}}^{\frac{1}{q}}\geq2,$
	$a_3^{\prime}[w]_{A_q}^{\frac{1}{qr}}\geq 2$.\par
	For this case, the result have the following representation by using $\Delta_2$ condition,
	\begin{equation*}
\begin{aligned}
		\int_{\mathbb{R}^{n}} \phi\left(|T_{\sigma, \Sigma \mathbf{b}}(\vec{f})(x)|\right) w(x) d x
\lesssim & [w]_{A_\infty}^{1+2\alpha C_1}[w]_{A_q}^{\frac{C_1m}{q}}\|\vec{b}\|_{\mathrm{BMO}}^{1+\alpha C_1}([w]_{A_\infty}+1) \\ &\times\left(\prod\limits_{i=1}^{m}\int_{\mathbb{R}^n}\phi^m(|f_i(x)|)w(x)dx \right)^{\frac{1}{m}}. \\
\end{aligned}
\end{equation*}
Combining the four cases, each of them has an exact expression for the constant of $w$, we have thus completed the proof of the theorem.
\end{proof}

For $\frac{1}{p}=\frac{1}{p_1}+\ldots +\frac{1}{p_m}$ with $1<p_1,\ldots, p_m<\infty,$ we note that Theorem \ref{thm1.6} cannot imply the classical $L^{p_1} \times\ldots \times L^{p_m}$ to $L^p$ boundedness when $\phi=t^p.$ In fact, by Theorem \ref{thm1.6} we have the following corollary consistent with the classical case.\par
\begin{corollary}\label{cor3}
Let $\sigma \in S_{\rho,\delta}^l(n,m)$ with $0 \leq \rho,\delta \leq 1, l < mn(\rho-1),$  $\vec{b} =\left(b_{1}, \ldots, b_{m}\right)\in \mathrm{BMO}^m.$ Assume that~$\phi$ is a~$N-$function with sub-multiplicative property. Let\\ $\frac{1}{p}=\frac{1}{p_1}+\ldots +\frac{1}{p_m}$ with $1<p_1,\ldots, p_m<\infty, $ then for each~$1< r < \infty,$ $r<i_{\phi}<\infty,$ there exist constants~$\alpha, C(\phi,w,r)$ such that for
every~$1<q<\frac{i_{\phi}}{r}$ and ~$w\in A_q ,$
\begin{equation*}
\begin{aligned}
\int_{\mathbb{R}^n}\phi(|T_{\sigma,\Sigma \mathbf{b}}(\vec{f})(x)|)w(x)dx\leq &C(\phi,w,r)\| \vec{b}\|_{\mathrm{BMO}}^{1+\alpha C_1}\\
&\times \left(\prod\limits_{i=1}^{m}\int_{\mathbb{R}^n}\phi^{\frac{p_i}{p}}(|f_i(x)|)w(x)dx \right)^{\frac{p}{p_i}}.
\end{aligned}
\end{equation*}\par
In particular, if $\phi(t)=t^p$ then
$$\|T_{\sigma,\Sigma \mathbf{b}}(\vec{f})\|_{L^p(w)}\lesssim C(\phi,w,r) \|\vec{b}\|_{\mathrm{BMO}}^{1+\alpha C_1}\prod\limits_{i=1}^{m}\|f_i\|_{L^{p_i}(w)}.$$
\end{corollary}
\begin{proof}
Notice that $\frac{1}{p}=\frac{1}{p_1}+\ldots +\frac{1}{p_m},$ by H\"{o}lder's inequality and Lemma \ref{lem4.1}, we see that
\begin{equation*}
\begin{aligned}
\|\mathcal{M}_w^{\mathcal{D}}(\vec{f}_\phi)\|_{L^{\frac{1}{\alpha}}(w)}^{\frac{1}{\alpha}}&\leq
\int_{\mathbb{R}^n} |\prod _{i=1}^mM_{w}^{\mathcal{D}}(f_{\phi,i})(x)|^{\frac{1}{\alpha}} w(x) d x \\
&\leq\prod _{i=1}^m\left(\int_{\mathbb{R}^n} |M_{w}^{\mathcal{D}}(f_{\phi,i})(x)|^{\frac{p_i}{p\alpha}} w(x) d x\right)^{\frac{p}{p_i}}\\
&\leq \prod_{i=1}^m{(\frac{p_i}{p\alpha})^{\prime}}^{\frac{1}{\alpha}}\left(\int_{\mathbb{R}^n} |f_{\phi,i}(x)|^{\frac{p_i}{p\alpha}}w(x)d x\right)^{\frac{p}{p_i}}.\\
\end{aligned}
\end{equation*}
This inequality, together with the estimate of (\ref{ie5.11}), indicates that
$$\int_{\mathbb{R}^n} \phi\left(\mathcal{M}_w^{\mathcal{D}}(\vec{f})(x)\right) w(x) d x \leq \prod_{i=1}^m{(\frac{p_i}{p\alpha})^{\prime}}^{\frac{1}{\alpha}}\left(\int_{\mathbb{R}^n} \phi^{\frac{p_i}{p}}(|f_i(x)|)w(x)d x\right)^{\frac{p}{p_i}}.$$
Using the method in Lemma \ref{lem4.2}, it further gives that
\begin{equation}\label{ie5.12}
\int_{\mathbb{R}^n} \phi\left(\mathcal{M}_r(\vec{f})(x)\right) w(x) d x \leq a_3\prod_{i=1}^m\left(\int_{\mathbb{R}^n} \phi^{\frac{p_i}{p}}(a_3[w]_{A_q}^{\frac{1}{qr}}|f_i(x)|)w(x)d x\right)^{\frac{p}{p_i}}.
\end{equation}
Then the proof process of Theorem \ref{thm1.6} and (\ref{ie5.12}) together imply the desired result.
\end{proof}

\section{ applications}
This section is devoted to giving some applications of our theorems for multilinear Fourier multiplier and multilinear square function. As a corollary of Theorem \ref{thm1.1}, we first give the following results of a class of the multilinear Calder\'{o}n-Zygmund singular integral operators.
\begin{corollary}
Let $T_K$ is a multilinear Calder\'{o}n-Zygmund singular integral operator with kernel function $K$ satisfies
$$
K(x, \vec{y})=\int_{\mathbb{R}^{m n}} \sigma(x, \vec{\xi}) e^{2 \pi i \vec{y} \cdot \vec{\xi}} d \vec{\xi},
$$
where $\vec{x}=(x, \ldots, x) \in \mathbb{R}^{m n},$ $\sigma \in S_{\rho, \delta}^r(n, m)$ with $\rho, \delta \in[0,1]$ and $r+\delta<n m(\rho-1)$ for some $0 \leq \delta \leq 1.$
Let $\mathbb{X},\mathbb{X}_i$ be rearrangement invariant Banach function spaces with
$1<p_{\mathbb{X}},p_{\mathbb{X}_i}\leq q_{\mathbb{X}}, q_{\mathbb{X}_i}<\infty$ for $i=1,2\cdots, m .$  Assume that $m$-product operator $P_m$ maps
$\overline{\mathbb{X}}_1\times \cdots \times\overline{\mathbb{X}}_m$ to $\overline{\mathbb{X}},$ then for every
$ w \in A_{\min\limits_{i}{\{p_{\mathbb{X}_i}\}}},$ $T_K $ is bounded from  $\X_1(w)\times \cdots \times \X_m(w)$ to $ \X(w).$
\end{corollary}

\begin{proof}
This corollary can be directly deduced from Theorem \ref{thm1.1} and \cite[Proposition 3.7]{cao}.
\end{proof}
\subsection{ Multilinear Fourier multiplier}
\par In this subsection, we apply our results to derive the boundedness of multilinear Fourier multipliers in rearrangement invariant Banach function spaces which contain the $L^p$ space.\par
The motivation for multilinear Fourier multipliers comes in large part from the fact that people began to consider a class of kernel conditions which are weaker than the H\"{o}lder continuity kernels, and consider the boundedness of some operators on the products of Lebesgue spaces (see \cite{kur}).  For the sake of simplicity, we only consider the bilinear case. Now we give the definition of bilinear Fourier multiplier.\par
 For integer $s$, we assume that $m \in C^s\left(\mathbb{R}^{2 n} \backslash\{0\}\right)$ satisfying the following condition:
\begin{equation}\label{con1}
\left|\partial_{\xi_1}^\alpha \partial_{\xi_2}^\beta m({\xi_1}, {\xi_2})\right| \leq C_{\alpha, \beta}(|{\xi_1}|+|{\xi_2}|)^{-(|\alpha|+|\beta|)}
\end{equation}
for all $|\alpha|+|\beta| \leq s$ and $({\xi_1}, {\xi_2}) \in \mathbb{R}^{2 n} \backslash\{0\}$. The bilinear Fourier multiplier operator $T_m$ is defined by
$$
T_m(f, g)(x)=\frac{1}{(2 \pi)^{2 n}} \int_{\mathbb{R}^n} \int_{\mathbb{R}^n} e^{i x \cdot({\xi_1}+{\xi_2})} m({\xi_1}, {\xi_2}) \hat{f}({\xi_1}) \hat{g}({\xi_2}) d {\xi_1} d{\xi_2}.
$$
It is clear that if $\sigma$ is independent of $x$ in multilinear pseudo-differential operator $T_\sigma$, that is $\sigma(x,\vec{\xi})=m(\vec{\xi}),$ then multilinear pseudo-differential operators are Fourier multipliers.\par
As in \cite[Theorem 5.10]{bui}, by Theorems \ref{thm1.1} and \ref{thm1.3}, we obtain the following Theorem \ref{thm5.1}. We omit the details.
\begin{theorem}\label{thm5.1}
Assume that (\ref{con1}) holds for some $n+1 \leq s \leq 2 n$. Let $2 n / s<p_0$ and $\mathbb{X} $ be a rearrangement invariant quasi-Banach function space with $p$-convex for some $0<p\leq 1.$ For $i=1,2$, let $\mathbb{X}_i$ be rearrangement invariant quasi-Banach function spaces with
$p_0<p_{\mathbb{X}},p_{\mathbb{X}_i}\leq q_{\mathbb{X}}, q_{\mathbb{X}_i}<\infty$ and each of them is $p_i$-convex for some $0<p_i \leq 1.$ If $P_m$ maps
$\overline{\mathbb{X}}_1\times \overline{\mathbb{X}}_2$ to $\overline{\mathbb{X}},$ then for every
$ w \in A_{\min\limits_{i}{\{p_{\mathbb{X}_i}\}}/{p_0}},$ there exist $q_1, q_2>1,$ such that

\begin{equation*}
\begin{aligned}
\left\| T_{m}(\vec{f})\right\|_{\mathbb{X}(w)} \lesssim [w]_{A_\infty}^{\frac{1}{p}}  \prod\limits_{i=1}\limits^2 [w]_{A_{p_{\mathbb{X}_i}/{p_0}}}^{\frac{1}{p_0q_i}} \left\| f_i\right\|_{\mathbb{X}_i(w)}.
\end{aligned}
\end{equation*}
\end{theorem}
\subsection{ Multilinear square functions}
This subsection is devoted to establishing modular inequality estimates for multilinear square functions, and we start with some simple definitions. \par
Let $\psi(x, \vec{y})$ be a locally integrable function defined away from the diagonal $x=y_1=$ $\ldots=y_m$ in $\mathbb{R}^{n \times(m+1)}$. We assume that $\psi(x, \vec{y})$ satisfies the following two size and smoothness conditions.
\par

Size condition: There exist $\delta ,A >0$ such that
$$
|\psi(x, \vec{y})| \leq \frac{A}{\left(1+\left|x-y_1\right|+\cdots+\left|x-y_m\right|\right)^{m n+\delta}}
$$\par
Smoothness conditions: There exist $\delta ,A, \gamma>0$ so that
$$
|\psi(x, \vec{y})-\psi(x+h, \vec{y})| \leq \frac{A|h|^\gamma}{\left(1+\left|x-y_1\right|+\cdots+\left|x-y_m\right|\right)^{m n+\delta+\gamma}},
$$
with $|h|<\frac{1}{2} \max\limits_j\left|x-y_j\right|$, and
\begin{small}
$$
\left|\psi\left(x, y_1, \ldots, y_i, \ldots, y_m\right)-\psi\left(x, y_1, \ldots, y_i+h, \ldots, y_m\right)\right| \leq \frac{A|h|^\gamma}{\left(1+\left|x-y_1\right|+\cdots+\left|x-y_m\right|\right)^{m n+\delta+\gamma}},
$$
\end{small}
with $|h|<\frac{1}{2}\left|x-y_i\right|$ for $i =1, \ldots, m$.\par
Let $\vec{f}=\left(f_1, \ldots, f_m\right) \in \mathcal{S}\left(\mathbb{R}^n\right) \times \cdots \times \mathcal{S}\left(\mathbb{R}^n\right)$ and $x \notin \bigcap_{j=1}^m \operatorname{supp} f_j,$ we define
$$
\psi_t(\vec{f})(x)=\frac{1}{t^{m n}} \int_{\left(\mathbb{R}^n\right)^m} \psi\left(\frac{x}{t}, \frac{y_1}{t}, \ldots, \frac{y_m}{t}\right) \prod_{j=1}^m f_j\left(y_j\right) d y_j
$$

According to $\psi(x, \vec{y}),$ we may define the multilinear square function $g_{\lambda, \psi}^*$ of the following form with $\lambda>2 m, \alpha>0.$
$$
g_{\lambda, \psi}^*(\vec{f})(x)=\left(\int_{\mathbb{R}_{+}^{n+1}}\left(\frac{t}{t+|x-y|}\right)^{n \lambda}\left|\psi_t(\vec{f})(y)\right|^2 \frac{d y d t}{t^{n+1}}\right)^{1 / 2}.
$$
Using \cite[Theorem 1.2]{bui1}, we obtain the modular inequality of the following form for $g_{\lambda, \psi}^*.$
\begin{theorem}\label{thm5.2}
Assume that $\lambda >2m.$  Let~$\phi$ be a~$N-$function with sub-multiplicative property. If $1<i_{\phi}<\infty$ then there exist constants~$\alpha, a_3,$ such that for
every~$1<q<{i_{\phi}}$ and ~$w\in A_q ,$
\begin{equation*}
\int_{\mathbb{R}^n}\phi(|g_{\lambda, \psi}^*(\vec{f})(x)|)w(x)dx\leq C(\phi,w)
 \left(\prod\limits_{i=1}^{m}\int_{\mathbb{R}^n}\phi^m(|f_i(x)|)w(x)dx \right)^{\frac{1}{m}},
\end{equation*}
and if $i_\phi =1, w\in A_1,$
$$\sup\limits_\lambda \phi(\lambda)w\left(\{x\in \mathbb{R}^n: |g_{\lambda, \psi}^*(\vec{f})(x)|>\lambda\}\right)^m\leq C\prod\limits_{i=1}^{m}\int_{\mathbb{R}^n}\phi(|f_i(x)|)w(x)dx,$$
where
\begin{equation*}
C(\phi, w)= \begin{cases}[w]_{A_{\infty}}^{1+\alpha C_{1}}, &  a_{3}[w]_{A_{q}}^{\frac{1}{q}}\leq 2, \\ [w]_{A_{\infty}}^{1+\alpha C_{1}}[w]_{A_{q}}^{\frac{mC_1}{q }}, & a_{3}[w]_{A_{q}}^{\frac{1}{q}}> 2.\end{cases}
\end{equation*}
\end{theorem}

\end{document}